\documentclass[a4paper,10pt]{amsart}
 \usepackage{amssymb,amsfonts,latexsym}
 \usepackage{graphics,verbatim}
 \usepackage{graphicx}
 \usepackage{comment}
 \usepackage{upref}
 \usepackage{hyperref}
 \usepackage{xcolor}
 \usepackage{amsthm}
 \usepackage{appendix}
 \usepackage{mathtools}
 \usepackage{thmtools}
\usepackage[shortlabels]{enumitem}
\mathtoolsset{showonlyrefs}
\allowdisplaybreaks

 \newtheorem{theorem}{Theorem}[section]
 \newtheorem{proposition}{Proposition}[section]
 \newtheorem{lemma}{Lemma}[section]
 \newtheorem{corollary}{Corollary}[section]
 \newtheorem{remark}{Remark}[section]

 \renewcommand{\(}{\left(}
 \renewcommand{\)}{\right)}
 \renewcommand{\[}{\left[}
 \renewcommand{\]}{\right]}
 \newcommand{\eps}{\varepsilon}
\newtheorem{estimate}{Estimate}[section]
 \newcommand{\grad}{\ensuremath{\nabla}}
 
 \newcommand{\N}{\ensuremath{\mathbb{N}}}
 \newcommand{\R}{\ensuremath{\mathbb{R}}}

 \newcommand{\B}{\mathbb{B}}

 \newcommand{\abs}[1]{\left\vert#1\right\vert}

 \newcommand{\norm}[1]{\left\Vert#1\right\Vert}

 \newcommand{\be} {\begin{equation}}
 \newcommand{\ee} {\end{equation}}
 \newcommand{\bea} {\begin{eqnarray}}
 \newcommand{\eea} {\end{eqnarray}}
 \newcommand{\Bea} {\begin{eqnarray*}}
 	\newcommand{\Eea} {\end{eqnarray*}}

 \newcommand{\ba} {\beta}

 \newcommand{\De} {\Delta}
 \newcommand{\la} {\lambda}

 \newcommand{\nequiv} {\not\equiv}

 \newcommand{\noi} {\noindent}
 \newcommand{\I}{\infty}

 \newcommand{\f}{\frac}

 \newcommand{\ef}{\eqref}
 
 \newcommand{\dvg}{\mathrm{~d}V_{\mathbb B^N}}
   \newcommand{\bn}{\mathbb B^N}
   \newcommand{\Rn}{\mathbb R^N}
 \newcommand{\cn}{c(N,\la)}

 \catcode`\@=11
 
 \makeatletter \@addtoreset{equation}{section} \makeatother

\begin{document}

 	\title[Poincar\'e-Sobolev equations]
 	{Poincar\'e-Sobolev equations with the critical exponent and a potential in the hyperbolic space}

\author[Bhakta]{Mousomi Bhakta}
\address{M. Bhakta, Department of Mathematics\\
Indian Institute of Science Education and Research Pune (IISER-Pune)\\
Dr Homi Bhabha Road, Pune-411008, India}
\email{mousomi@iiserpune.ac.in}
\author[Ganguly]{Debdip Ganguly}
\address{D. Ganguly, Department of Mathematics\\
Indian Institute of Technology Delhi\\
IIT Campus, Hauz Khas, New Delhi, Delhi 110016, India}
\email{debdipmath@gmail.com, debdip@maths.iitd.ac.in}	
\author[Gupta]{Diksha Gupta}
\address{D. Gupta, Department of Mathematics\\
Indian Institute of Technology Delhi\\
IIT Campus, Hauz Khas, New Delhi, Delhi 110016, India}
\email{maz208233@maths.iitd.ac.in}
\author[Sahoo]{Alok Kumar Sahoo}
 \address{A.K.Sahoo, Department of Mathematics\\
Indian Institute of Science Education and Research Pune (IISER-Pune)\\
Dr Homi Bhabha Road, Pune-411008, India}
\email{alok.sahoo@acads.iiserpune.ac.in}

 	
 	\date{}
 	 	
 	\begin{abstract}
 	On the hyperbolic space, we study a semilinear equation with non-autonomous nonlinearity having a critical Sobolev exponent. 
 The Poincar\'e-Sobolev equation on the hyperbolic space explored by Mancini and Sandeep [Ann. Sc. Norm. Super. Pisa Cl. Sci. 7 (2008)] resembles our equation.   As seen from the profile decomposition of the energy functional associated with the problem, the concentration happens along two distinct profiles: localised Aubin-Talenti bubbles and hyperbolic bubbles. Standard variational arguments cannot obtain solutions because of nontrivial potential and concentration phenomena. As a result, a deformation-type argument based on the critical points at infinity of the associated variational problem has been carried out to obtain solutions for $N>6.$ Conformal change of metric is used for proofs, enabling us to convert the original equation into a singular equation in a ball in $\mathbb{R}^N$ and perform a fine blow-up analysis. 
 
 	\end{abstract}
 	
 	\maketitle
 	\tableofcontents

 	
 	\section{Introduction}

Consider the following elliptic problem on the hyperbolic space of ball model $\mathbb B^N$:
 	\begin{equation}
 	\tag{$P_{a}$}\label{Pa}
 	\left\{\begin{aligned}
 	 	 -\De_{\mathbb B^N} u -\lambda u&=a(x) \abs{u}^{2^*-2}u  \quad\text{in } \mathbb B^N,\\
 	u>0 \text{ in } \mathbb B^N &\,\,\,\text{ and }\quad  u\in H^1(\mathbb B^N),
 	\end{aligned}
 	\right.
 	\end{equation}
 	where $2^*:=\f{2N}{N-2}$ and $\frac{N(N-2)}{4}<\lambda<\f{(N-1)^2}{4}$ and $N\geq 4$. $H^1(\bn)$ denotes the Sobolev space on the ball model of the hyperbolic space $\bn$, $\Delta_{\bn}$ denotes the
Laplace-Beltrami operator on $\bn$, $\f{(N-1)^2}{4}$  being the bottom of the $L^2$ spectrum of $-\De_{\bn}$.

 Throughout the article, we assume the following:
 \begin{equation}\tag{\textbf{A}}\label{a_maincond}
     0<a\in C^2(\mathbb B^N) \quad\text{ and }\quad a(x)\to 1 \text{ as } d(x,0)\to \infty, 
     \end{equation}
where $d(x,0)$ is the geodesic distance of $x$ from $0$ in the hyperbolic space $\bn$.  The Euclidean unit ball $B^N:=\{x\in\Rn:|x|<1\}$ endowed with the metric 
  $g_{\bn}:= \left(\frac{2}{1-|x|^2}\right)^2\;g_{Eucl}$ represents the ball model for the hyperbolic N-space where $g_{Eucl}$ is the Euclidean metric. The corresponding volume element is given by $\dvg=(\frac{2}{1-|x|^2})^N\mathrm{~d}x$, where $\mathrm{~d} x$ is the Lebesgue measure on $\Rn$. 
  
The geometric aspect of the problem is the primary feature of this equation, aside from the existence of the Sobolev exponent in \eqref{Pa}.  Mathematically speaking, the equation posed on the hyperbolic space presents some new features for the problem's analysis, which we will address in more detail below. Non-compact variational elliptic problems have garnered significant attention recently due to their applicability in various domains, including applied mathematics, differential geometry, and mathematical physics. Noncompactness is a common feature of equations involving critical nonlinearities, singularities, or unbounded domains, which presents a major obstacle for conventional methods. Many mathematicians have created intricate mathematical frameworks and analytical methods suited for non-compact problems to address these difficulties. These include the Palais-Smale (PS) condition, blow-up analysis, concentration-compactness principles, and numerous topological techniques. The theory of non-compact issues has been thoroughly studied and developed to a high degree of sophistication by many researchers. Brezis and Nirenberg \cite{BN} showed the existence and non-existence of positive solutions in bounded domains for the case of $\Rn$ in their key works.  Their novel approach to restoring compactness centres on the examination of local PS sequences. In addition, throughout the past few decades, many variants of scalar field equations with potential $a(\cdot)$ have been studied in detail, beginning with the groundbreaking work of Berestycki-Lions, Bahri-Li, Bahri-Berestycki, and Bahri-Li  \cite{BB1, Bahri-Li, Bahri, BL1, BL2}.  These subjects have now undergone extensive development, especially concerning the multiplicity-related concerns and the assumptions on the potential $a(\cdot)$. Works like \cite{Ad, AT2, CG1, CG2, CG3, CG4, CW}, among others, demonstrate this advancement; nonetheless, this list is by no means complete.

With its connections to several other critical exponent problems, such as the Hardy-Sobolev-Maz'ya equations, and several geometric partial differential equations, including those found in the Yamabe problem, a natural generalisation of this well-known Brezis-Nirenberg problem in the hyperbolic space gained prominence. Mancini and Sandeep thoroughly studied the existence and uniqueness of positive solutions to this Brezis-Nirenberg-type problem in hyperbolic space. They have shown in their seminal work \cite{MS} that the following equation
\begin{equation}
 	\tag{$P_{\lambda,1,0}$}\label{E}
 	\left\{\begin{aligned}
 	 	 -\De_{\mathbb B^N} u -\lambda u&=\abs{u}^{2^*-2}u \quad\text{in } \mathbb B^N,\\
 	u>0 \text{ in } \mathbb B^N &\,\,\,\text{ and } u\in H^1(\mathbb B^N)
 	\end{aligned}
 	\right.
 	\end{equation}
  has a unique weak solution $\mathcal{U}:=\mathcal{U}_{N,\la}$ (up to hyperbolic isometries), which is also an extremal of the Poincar\'e-Sobolev inequality \eqref{S-inq}. By elliptic regularity, any solution of \eqref{E} is also $C^\infty$. In \cite{MS}, authors have also shown that any solution of \eqref{E} is radially symmetric, symmetric-decreasing  and satisfies the following decay property:  for every $\varepsilon > 0,$ there exist positive constants $C_1^{\varepsilon}$ and $C_2^{\varepsilon},$ depending on $\varepsilon,$ such that 
  \begin{equation*}
	C_1^{\varepsilon}\;e^{-(\cn + \varepsilon) \, d(x,0)} \leq \mathcal{U}(x) \leq C_2^{\varepsilon}\; e^{-(\cn - \varepsilon) \, d(x,0)}, \ \ \ \mbox{for all} \ \ x \in \bn,
\end{equation*}
where $\cn=\frac{N-1+\sqrt{(N-1)^2-4\la}}{2}$ and $d(x,0)$ is the hyperbolic distance. Later, Bandle and Kabeya in \cite[Lemma~2.3]{CYK} improved the above result and showed that the estimate holds for  $\varepsilon =0,$ i.e.,  there exist constants $\chi_i:= \chi_i(N,\la),\,  i = 1,2$  such that 
\begin{align} \label{bubble decay}
\chi_1e^{-\cn d(x,0)} \leq \mathcal{U}(x) \leq \chi_2 e^{- \cn d(x,0)}, \ \ \ \mbox{for all} \ \ x \in \bn.
\end{align}
For simplicity of notations, we will always denote the radially symmetric
solution by $\mathcal{U}$ and the dependence of $N,\la$ will be implicitly assumed. Moreover, the authors in \cite{MS} discovered that \eqref{E} naturally arises when studying the Euler-Lagrange equations that correspond
to the Hardy-Sobolev-Maz'ya (HSM) inequalities. Recently, mathematicians have been curious to investigate equivalent HSM inequalities for higher-order derivatives after its connection with \eqref{E} (see \cite{LY1, LY2}). The
authors in \cite{LLY} have thereafter studied the existence, nonexistence, and symmetry of solutions
to the higher-order Brezis-Nirenberg problem in the hyperbolic space.  The aforementioned authors' work depends on the Hardy-Littlewood-Sobolev inequality on the hyperbolic space, as well as on estimates of Green's functions for the kernels of powers of fractional Laplacian and the Helgason-Fourier analysis.

Inspired by this, in this article, we aim to investigate whether the solutions persist in the critical case after perturbing with a non-radial potential function, denoted by $a(x)$. The primary difficulty in studying such problems across the entirety of $\bn$ lies in the translation invariance of the limit equation \eqref{E} under the set of isometries of 
$\bn$, resulting in the loss of compactness. On the other hand, the local compactness argument and accurate energy estimates can be exploited in the subcritical scenario (see \cite{GGS1, GGS2, GGS3}) for the existence of solutions. Nonetheless, in the presence of the critical exponent, even local compactness cannot be regained. Indeed, the PS decomposition established by Bhakta and Sandeep in \cite{BS} identified the precise cause of this loss of compactness, demonstrating that it can be attributed to two profiles: one along the hyperbolic bubble, i.e., the solutions of \eqref{E}, and the other along the localized Aubin-Talenti profile, i.e., the positive solutions of
\be\label{Talenti}
-\De V =  \abs{V}^{2^*-2}V,  \quad V\in D^{1,2}(\mathbb{R}^N). 
\ee
Positive solutions of the above equation are given by
\begin{equation*}
V^{\varepsilon, y}(x):=\varepsilon^{(2-N) / 2} V\left(\frac{x-y}{\varepsilon}\right),
\end{equation*}
where
\begin{equation*}
V(x):=[N(N-2)]^\frac{N-2}{4}\left(1+|x|^2\right)^{\frac{2-N}{2}}.
\end{equation*}

Furthermore, we investigated the problem \eqref{Pa} with a non-homogeneous term $f \not\equiv 0$ in \cite{BGGS}. To establish solutions in that scenario, we initially extended the PS decomposition of Bhakta-Sandeep, which posed significant challenges due to the presence of the potential $a(\cdot)$ and necessitated several estimates and geometric arguments regarding the isometry group (Möbius group) of the hyperbolic space. However, the techniques employed to prove the existence results would not be applicable here because $f \equiv 0.$ Indeed, for case $a(x)\leq 1$, the first solution was identified as a unique critical point $\mathcal{V}_{a, f}(x)$ of the corresponding energy functional $I_{\la, a, f}$ in  $B\left(r_{1}\right)=\left\{u \in H^{1}\left(\mathbb{B}^{N}\right)\right.$: $\left.\|u\|<r_{1}\right\}$ for some $r_1>0$.
 Then, we proved the existence of $R>0$ such that the following subadditivity property is demonstrated by the energy functional
	\begin{equation}
		I_{\la,a, f}\left(\mathcal{V}_{a, f} (x)+t\;\mathcal{U}(\tau_{-y}(x))\right)<I_{\la,a, f}\left(\mathcal{V}_{a, f}(x)\right)+I_{\la,1,0}(\mathcal{U})\;\;\; \label{4.ab}
	\end{equation} 
for all $d(y,0) \geq R$, $t>0$, $\|f\|_{H^{-1}\left(\mathbb{B}^{N}\right)}$ sufficiently small, and $\tau_{-y}(.)$ denotes the hyperbolic translation with $\tau_{-y}(0)=-y$. Now note that $\mathcal{V}_{a, 0}(x)=0$ and hence the energy estimate does not hold when $f \equiv 0$, and therefore the entire arguments, when applied to the case $f \equiv 0$, collapses. 
Furthermore, in the case where $a(x)\geq 1$, we partition $H^1(\bn)$ into three components and demonstrate that the energy functional achieves its infimum on one of these components. The proof hinges on the existence of a certain $u$ belonging to one of the components such that $\langle f, u \rangle >0$. Subsequently, the derivation of the second solution relies on the first solution and a similar energy estimate as \eqref{4.ab}, aiding us in establishing the min-max level below the first level of the breaking of the PS sequence. Thus, the case addressed in this article, where $f\equiv 0$, presents its own set of challenges and requires an entirely different approach.

As expected, we turn to the PS decomposition (see Theorem \ref{ps_decom}) to address the loss of compactness even in this case. We then use this global compactness result to prove existence results. Nevertheless, accomplishing this is challenging, so we divide our study into three cases based on the levels of the breakdown of the PS condition. Initially, we address the scenarios where the Aubin-Talenti blow-up can be circumvented, specifically when $2D_\lambda \leq \max(a)^{\frac{2-N}{2}}D$ or $\max(a)^{\frac{2-N}{2}}D \leq D_\lambda$, where $D_\lambda$ and $D$ are defined in \eqref{Energies}. In these instances, we establish the existence of solutions by constructing a PS sequence at levels within these safe energy ranges, employing the Ekeland variational principle under different sign relations between the potential term and its behaviour at infinity. The challenge arises when the energy level of the PS sequence potentially falls within the region where the Aubin-Talenti blow-up occurs. To address this, we exploit \it  Critical point at infinity \rm due to A.~Bahri and then carefully analyse the problem, we refer, 
(\cite{BahriB, BB, BC2, Bahri-Li, BE1, BE, BN,  CFS,  KW,  Li1,Li2, Li-Zhu,Sch, Struwe, Yam}) for more details. 

The structure of the paper is as follows: Section \ref{pre} covers the preliminaries, where we introduce essential notations, geometric definitions, and fundamental results in hyperbolic space. Section \ref{WithPS} presents the existence results, focusing on scenarios where Palais-Smale sequences can be safely constructed at specific energy levels. In Section \ref{WithoutPS}, we address the remaining case where the energy level of the constructed Palais-Smale sequence may fall within the region where an Aubin-Talenti blow-up cannot be excluded. Section \ref{WithoutPS} is devoted to the proof of Theorem \ref{theorem-critical-point-infinity}. Section \ref{appen} provides the appendix, which contains crucial estimates used in the proof of Theorem \ref{theorem-critical-point-infinity}.

\medskip

{\bf Notations:} Throughout this paper, we denote by $\mathcal{U}$ as the unique positive solution of \eqref{E}. $S_{\la}$ denotes the best constant in \eqref{S-inq} and $S$ denotes the best constant in Sobolev inequality in $\Rn$. $\tau_b(.)$ denotes the hyperbolic translation (see \eqref{hyperbolictranslation}) with $\tau_b(0)=b$. Numerous positive constants, whose specific values are irrelevant, are denoted by $C$.

    \section{Preliminaries}\label{pre}
   
 We denote the inner product on the tangent space of $\bn$ by $\langle,\rangle_{\bn}$. $\nabla_{\bn}$ and $\Delta_{\bn}$ denote gradient 
 vector field and Laplace-Beltrami operator, respectively. Therefore in terms of local (global) coordinates $\nabla_{\bn}$ and $\Delta_{\bn}$ takes the form
\begin{align*} 
 \nabla_{\bn} = \left(\frac{1 - |x|^2}{2}\right)^2\nabla,  \quad 
 \Delta_{\bn} = \left(\frac{1 - |x|^2}{2}\right)^2 \Delta + (N - 2)\left(\frac{1 - |x|^2}{2}\right)  x \cdot \nabla,
\end{align*}
where $\nabla, \Delta$ are the standard Euclidean gradient vector field and Laplace operator, respectively, and '$\cdot$' denotes the 
standard inner product in $\mathbb{R}^N.$

 \medskip 

\noindent

 {\bf Hyperbolic distance on $\bn.$} The hyperbolic distance between two points $x$ and $y$ in $\bn$ will be denoted by $d(x, y).$ The hyperbolic distance between
$x$ and the origin can be computed explicitly  
\begin{align}\label{hyp-dist}
\rho := \, d(x, 0) = \int_{0}^{|x|} \frac{2}{1 - s^2} \, {\rm d}s \, = \, \log \frac{1 + |x|}{1 - |x|},
\end{align}
and therefore  $|x| = \tanh \frac{\rho}{2}.$ Moreover, the hyperbolic distance between $x, y \in \bn$ is given by 
\begin{align*}
d(x, y) = \cosh^{-1} \left( 1 + \dfrac{2|x - y|^2}{(1 - |x|^2)(1 - |y|^2)} \right).
\end{align*}
As a result, a subset of $\bn$ is a hyperbolic sphere in $\bn$ if and only if it is a Euclidean sphere in $\mathbb{R}^N$ and contained in $\bn$, possibly 
with a different centre and different radius, which can be explicitly computed from the formula of $d(x,y)$ \cite{RAT}.  Geodesic balls in $\bn$ with radius $r$ centered at $x \in \bn$ will be denoted by
$$
B_r(x) : = \{ y \in \bn : d(x, y) < r \},
$$
and $B(x,r)$ denotes the open ball in Euclidean space with center $x$ and radius $r > 0$.

\medskip
\noindent
Next, we introduce the concept of hyperbolic translation.

\textbf{Hyperbolic translation.} Given a point $a \in \Rn$ such that $|a|>1$  and $r>0,$ let $S(a,r):=\{x \in \Rn \ | \ |x-a| = r\}$ be the sphere in $\Rn$ with center $a$ and radius $r$ that intersects $S(0,1)$ orthogonally. It is known that it is the case if and only if $r^2 = |a|^2 -1,$ and hence $r$ is determined by $a.$ Let $\rho_{a}$ denotes the reflection with respect to the plane $P_a := \{x \in \Rn \ | \ x\cdot a = 0\}$ and $\sigma_a$ denotes the inversion with respect to the sphere $S(a,r).$ Then $\sigma_a\rho_a$ leaves $\bn$ invariant (see \cite{RAT}).
 
 For $b \in \mathbb{B}^N,$ the hyperbolic translation $\tau_{b}: \mathbb{B}^N \rightarrow \mathbb{B}^N$ that takes $0$ to $b$ is defined by $\tau_b = \sigma_{b^*}\circ\rho_{b^*}$ and can be expressed by the following formula
 
  \begin{align} \label{hyperbolictranslation}
  \tau_{b}(x) := \frac{(1 - |b|^2)x + (|x|^2 + 2 x \cdot b + 1)b}{|b|^2|x|^2 + 2 x\cdot b  + 1},
 \end{align}
where $b^* = \frac{b}{|b|^2}.$
 It turns out that $\tau_{b}$ is an isometry and forms the M\"obius group of $\bn$ (see \cite{RAT}, Theorem 4.4.6 for details and further discussions on isometries). Note that $\tau_{-b} = \sigma_{-b^*}\circ\rho_{-b^*}$ is the hyperbolic translation that takes $b$ to the origin. In other words, the hyperbolic translation that takes $b$ to the origin is the composition of the reflection $\rho_{-b^*}$ and the inversion $\sigma_{-b^*}.$

  Now we recall the Poincar\'e-Sobolev inequality: for $N\geq 3$ and $\la<\frac{(N-1)^2}{4}$, there exists an optimal constant $S_{\lambda} >0$ such that
  \begin{equation}
	S_{\lambda}\Big(\int_{ \mathbb B^N}|u|^{2^*}\dvg\Big)^\frac{2}{2^*}\leq \int_{ \mathbb B^N}\Big[|\nabla_{\mathbb B^N}u|^2-\lambda u^2\Big]\dvg \label{S-inq}
 \end{equation}
	for all $u\in C_c^\I(\mathbb{B}^N)$. Furthermore, the constant $S_{\lambda}$ is achieved by a unique positive solution $\mathcal U$ of \eqref{E} (i.e., unique solution of  \eqref{Pa} with $a\equiv 1$) and $\mathcal U$ is hyperbolic symmetric (see \cite{MS}). By  $S$, we denote the best constant in the Sobolev inequality in $\Rn$, i.e.
	$$  S \Big(\int_{ \mathbb R^N}|u|^{2^*}\mathrm{~d}x\Big)^\frac{2}{2^*}\leq \int_{ \mathbb R^N} |\nabla  u|^2\mathrm{~d}x$$
	for all $u\in C_c^\I(\mathbb{R}^N)$. Moreover $S_{\lambda}<S$ for $\lambda>\frac{N(N-2)}{4}$ (See \cite{AX14618}). It's well known that $S$ is achieved by the unique (up to translation and dilation) positive solution of \eqref{Talenti}. The associated energy $J(V)$ is given by
$$J(V)=\frac{1}{2}\int_{\mathbb{R}^N}|\nabla V|^2\mathrm{~d}x-\frac{1}{2^*}\int_{\mathbb{R}^N}|V|^{2^*}\mathrm{~d}x.$$
Corresponding to the equation \ef{Pa}, we define the energy functional $\mathcal I_{a}:H^1(\mathbb B^N)\to\R$ by
 	 	\be
 	 	\mathcal I_{a}(u)=\f{1}{2}\norm{u}^2_\la-\f{1}{2^*}\int_{\mathbb B^N}a(x)|u|^{2^*}\dvg(x),
 	 	\ee
 where $$\|u\|_{\lambda}^2 : = \int \limits_{\B^{N}}\left(|\nabla_{\B^{N}}u|^{2}-\lambda u^{2}\right) \dvg~\hbox{ for all } u\in H^1(\B^{N}),$$
is an equivalent norm w.r.t. the standard norm in $H^1(\mathbb{B}^N)$, since $\lambda < \frac{(N-1)^2}{4}$ and the corresponding inner product is given by $\langle \cdot, \cdot\rangle_{\lambda}.$ Denote 
\begin{equation}
D_\lambda:=\mathcal I_{1}(\mathcal U) \quad\text{and}\quad D:=J(V). \label{Energies}\end{equation}

Therefore, $\mathcal I_{1}(\mathcal U)=\frac{1}{N}S_\lambda^\frac{N}{2}$ and $J(V)=\frac{1}{N}S^\frac{N}{2}$, hence  $D_\lambda<D$.
 
 \medskip

 A sequence $(u_n)_n$ in $H^1(\mathbb B^N)$ is said to be a Palais-Smale sequence (in short, PS sequence) of $\mathcal I_{a}$ at the level $c$ if  $\mathcal I_{a}(u_n)\to c$ and $	\mathcal I^{\prime}_{a}(u_n)\to 0$ in $H^{-1}(\mathbb B^N)$.  Here we recall the profile decomposition theorem associated with the functional $\mathcal I_{a}$ from \cite[Theorem 3.1]{BGGS}.

\begin{theorem}\label{ps_decom}\cite[Theorem 3.1]{BGGS}
 		Let $(u_n)_n\in H^1(\mathbb B^N)$ be a PS sequence at the level $\ba\geq 0$. There exists  $n_1,\, n_2\in \N$,  a critical point $u$ of $\mathcal I_{a}$, and sequences $u^j_n,\, v^k_n\in H^1(\mathbb B^N)$ for all $1\leq j\leq n_1$ and $1\leq k\leq n_2$, such that  up to a subsequence (still denoted by the same index),
   $$u_n=u+\sum_{j=1}^{n_1}u_n^j+\sum_{k=1}^{n_2}v_n^k+o(1),$$
 	where $o(1)\to 0$ in $H^1(\mathbb B^N)$. The sequences $(u_n^j)_{n\geq 1}$ are PS sequences of the form $u_n=\mathcal{U}\circ T_n$ and $T_n$ is a hyperbolic isometry on $\bn$ such that $T_n(0)\to\infty$. Let $y_0\in\mathbb{B}^N$, and $\phi \in C_c^{\infty}\left(\mathbb{B}^N\right)$ such that $0 \leq \phi \leq 1,\; \phi(x)=1$ for $|x|<r$ and $\phi(x)=0$ for $|x|>R$, where $R$ is chosen so that  $B(y_0,R) \subseteq B(0,1)$. 
 The sequences $(v_n^k)_{n\geq 1}$ are PS sequences of the form
\be\label{seq-vn}
v_{n}(x):=\Big(\f{1-|x|^2}{2}\Big)^\f{N-2}{2}a(y_0)^{\frac{2-N}{4}}\phi(x-y_0)  \eps_n^{\f{2-N}{2}}V\(\f{x-y_0}{\eps_n}\),
\ee
  where $\eps_n>0$ and $\eps_n\to0$.
 Moreover, 
	$$\ba=\mathcal I_{a}(u)+\sum_{j=1}^{n_1}\mathcal I_{1}(\mathcal U^j)+\sum_{k=1}^{n_2} a(y^k)^{-\frac{N-2}{2}}J(V^k)+o(1),$$
where $ \mathcal U^j$ and $V^k$ are the solutions of \eqref{E} and \eqref{Talenti} respectively corresponding to $u^j_n$ and $v_n^k$. Furthermore, we have
\begin{enumerate}
	\item $T^j_n \circ T^{-i}_n(0)\to \I $ as $n\to \I$ for $i\ne j$, where $T^{-i}_n= (T^{i}_n)^{-1}$.
	\item $\Big| \log\big(\frac{\eps^k_{n} }{\eps^l_{n}}\big)\Big|+ \Big| \frac{y_n^k-y^l_n}{\eps^l_{n}}\Big|\to \I$ as $n\to \I$ for $k\ne l$.
\end{enumerate}
\end{theorem}
\begin{corollary}
    For arbitrary naturals $n_1,n_2$ and points $y^k\in\mathbb R^N$, any positive PS-sequence for $\mathcal I_{a}$ at a level that cannot be written in the following form $$\sum_{j=1}^{n_1}\mathcal I_{1}(\mathcal U^j)+\sum_{k=1}^{n_2} a(y^k)^{-\frac{N-2}{2}}J(V^k)$$  yields a nontrivial weak solution of \eqref{Pa}.
\end{corollary}
\begin{proof}
    Let $u_n$ be a PS sequence at the level $\beta$. Then $u_n$ is bounded, and suppose its weak limit is $u$. Now if $u\not\equiv0$, then $u$ is a weak solution. Now if $u\equiv 0$, then using the previous theorem, we have $$\ba= \sum_{j=1}^{n_1}\mathcal I_{1}(U^j)+\sum_{k=1}^{n_2} a(y^k)^{-\frac{N-2}{2}}J(V^k).$$ This is a contradiction. This proves the corollary.
\end{proof}

\section{Existence of solutions: The case with Palais-Smale condition }\label{WithPS}
This section discusses the results in the region where the Palais-Smale (PS) condition holds. Initially, we determine these ranges based on $\lambda$.
 It is challenging to compare $a(y^k)^{\frac{2-N}{2}}D$ with $D_\lambda$, even though $D_\lambda <D$. The domain will now be divided into the following three cases for analysis:

    \begin{itemize}
        \item[$(a)$] \label{casea}$ 2D_\lambda\leq \max (a)^{\frac{2-N}{2}}D,$
        \item [$(b)$] $ D_\lambda< \max (a)^{\frac{2-N}{2}}D<2D_\lambda,$
        \item [$(c)$] $ \max (a)^{\frac{2-N}{2}}D\leq D_\la.$
    \end{itemize}
    
  Whether the solutions are constrained minimizers (index 1 solutions) or constrained mountain pass points (index 2 solutions) will depend on how the weight function $a(.)$ behaves. \\
  
  \smallskip
  \noindent
\textbf{Part I: Constrained Minimizers.} We define the Nehari manifold $\mathcal N$  by
\be\label{NeMan}
\mathcal N:=\left\{u\in  H^1(\mathbb B^N)\setminus\{0\}\,\,\text{ such that }\,\,\langle\,\mathcal I_a'(u),u\,\rangle=0  \right\}. 
\ee
Let $u\not\equiv 0$, then there exists a $t_u$ such that $t_uu\in\mathcal N$. A simple computation reveals 
\begin{equation}
    t_u=\left(\frac{\norm{u}^2_\la}{\int_{\mathbb B^N}a(x)|u|^{2^*}\dvg(x)}\right)^{\frac{1}{2^*-2}}, \label{scaNeh}
\end{equation}
and \begin{equation*}
    \mathcal{I}_a(t_uu)=\frac{1}{N}\norm{u}^N_\la      \left(\int_{\mathbb B^N}a(x)|u|^{2^*}\dvg(x)\right)^{-\frac{N}{2^*}}.
    \end{equation*}
Therefore, we can write 
 \begin{align*}
        \mathcal{I}_a(t_uu)&= \frac{1}{N}\left( E_a(u)\right)^{\frac{N}{2}}
    \end{align*}
    where $$E_a(u):= \norm{u}^2_\la      \left(\int_{\mathbb B^N}a(x)|u|^{2^*}\dvg(x)\right)^{-\frac{2}{2^*}}
    \quad \mbox{and} \quad E_\I(u):= \norm{u}^2_\la      \left(\int_{\mathbb B^N} |u|^{2^*}\dvg(x)\right)^{-\frac{2}{2^*}}.$$

\medskip

\begin{proposition}\label{Pro_N_inf}
    Assume \eqref{a_maincond} holds. If $$ \inf_{u\in\mathcal N}\mathcal I_a(u)\,<\,\min \Big( D_\lambda\,,\,\, \max (a)^{\frac{2-N}{2}}D\Big),$$ then the infimum is achieved. Moreover, $\mathcal I_a$ has a nontrivial critical point.
\end{proposition}
\begin{proof}
    Using the Ekeland variational principle, there exists a positive PS sequence $(u_n)_n$ for $\mathcal I_a$ (without loss of generality, we may assume that the minimizing sequence is positive)  restricted to $\mathcal N$ at the infimum level. Since $(u_n)_n$ is a PS sequence for $\mathcal I_a$ by definition of $\mathcal N,$ $u_n\rightharpoonup u\in \mathcal N$. $(u_n)_n$ has a strongly convergent subsequence that converges to some minimizer, according to Theorem \ref{ps_decom}. The minimizer is a nontrivial critical point of $\mathcal I_a$, using features of $\mathcal N$.
\end{proof}

We will explore different possibilities to verify the hypothesis as in Proposition \ref{Pro_N_inf}.

\begin{theorem}
    Let $N \geq 4.$ Assume \eqref{a_maincond} and we are in either case $(a)$ or $(b)$, and additionally, one of the following conditions holds:
\begin{enumerate}
       \item[]$(i)$  $a(x)\geq 1$ for all $x\in \mathbb B^N$.
       \medskip
     \item[]$(ii)$  $\int_{\mathbb B^N}\left[a(x)-1\right]\mathcal U^{2^*}\dvg>0.$ 
       
   \end{enumerate}    
\noi  Then \eqref{Pa} admits at least one positive solution.
\end{theorem}

\begin{proof}
    We know the existence result for $a(x)\equiv 1$. So now we assume $a(x)\not\equiv 1$. Moreover, since either of $(i)$ and $(ii)$ holds, there exists $x_0\in\mathbb B^N$ such that $a(x_0)> 1$. Thus $\max(a)>1$. We are in the case $(a)$ or $(b)$ and consequently, $ \min \Big( D_\lambda\,,\,\, \max (a)^{\frac{2-N}{2}}D\Big)=D_\lambda$. Therefore, thanks to Proposition~\ref{Pro_N_inf},  to conclude the proof it is enough to show that    $\inf_{u\in\mathcal N}\mathcal I_a\,<\,D_\lambda$.
 
\medskip
\noindent

\textbf{Case 1:} when  $a(x)\geq 1
\;\;\forall x\in\bn$, we have
    \begin{align*}
        E_{a}(\mathcal U)&= \frac{\norm{\mathcal U}^2_\la }{     \left(\int_{\mathbb B^N}a(x) |\mathcal U|^{2^*}\dvg(x)\right)^{\frac{2}{2^*}}}\leq \frac{\norm{\mathcal U}^2_\la }{     \left(\int_{\mathbb B^N}  |\mathcal U|^{2^*}\dvg(x)\right)^{\frac{2}{2^*}}}=E_\I(\mathcal U).
    \end{align*}
    Now if $E_{a}(\mathcal U)=E_\I(\mathcal U) $, then
    \begin{align*}
        \int_{\mathbb B^N}a(x)  |\mathcal U|^{2^*}\dvg=\int_{\mathbb B^N}  |\mathcal U|^{2^*}\dvg\implies \int_{\mathbb B^N}\Big(a(x)-1\Big)  |\mathcal U|^{2^*}\dvg=0.
    \end{align*}
    This is a contradiction because $\mathcal U>0$ and $a(x)\nequiv1 $. Hence $E_{a}(\mathcal U)<E_\I(\mathcal U) $. 
    
   \medskip 
    \noindent
    
\textbf{Case 2:}  If $(ii)$ holds, then 
    \begin{align*}
              E_{a}(\mathcal U)= \frac{\norm{\mathcal U}^2_\la }{     \left(\int_{\mathbb B^N}a(x) |\mathcal U|^{2^*}\dvg(x)\right)^{\frac{2}{2^*}}}&=\frac{\norm{\mathcal U}^2_\la }{     \left(\int_{\mathbb B^N}  |\mathcal U|^{2^*}\dvg+ \int_{\mathbb B^N}(a(x)-1)  |\mathcal U|^{2^*}\dvg(x)\right)^{\frac{2}{2^*}}}\\&<\frac{\norm{\mathcal U}^2_\la }{     \left(\int_{\mathbb B^N}  |\mathcal U|^{2^*}\dvg\right)^{\frac{2}{2^*}}}=E_\I(\mathcal U).
    \end{align*}
    Hence in both the cases $E_{a}(\mathcal U)<E_\I(\mathcal U) $. Now, as we observed before, $t_\mathcal{U}\mathcal U\in \mathcal N$ and
 $$\mathcal I_a(t_\mathcal U\mathcal U) =\frac{1}{N}\left( E_a(\mathcal U)\right)^{\frac{N}{2}}< \frac{1}{N}\left( E_\infty(\mathcal U)\right)^{\frac{N}{2}}=\frac{1}{N}S_\lambda^{\frac{N}{2}}=\mathcal I_{1}(\mathcal U)=D_\lambda,$$
which in turn implies $\inf_{u\in\mathcal N}\mathcal I_a\,<\,D_\lambda$. This completes the proof. 
\end{proof}

The case $(c),$ that is, $ \max (a)^{\frac{2-N}{2}}D\leq D_\la,$ will now be examined. To achieve this, we will convert the problem into a singular problem in the Euclidean ball using the conformal change of metric.

\medskip

\textbf{Conformal Change of Metric:} It is well known that the conformal Laplacian or the Yamabe operator $P_{1,\bn}:= -\Delta_{\bn} + \frac{(N-2)}{4(N-1)}R_{\bn} = -\Delta_{\bn} - \frac{N(N-2)}{4}$ is the first order conformal invariant operator where $R_{\bn}:= -N(N-1)$ is the scalar curvature with respect to the metric $g_{\bn}$. This implies that if  $\tilde g = e^{2\psi}g$ is a conformal metric then $P_{1,\tilde g} (u) = e^{-(\frac{N}{2} +1)\psi} P_{1,\bn}(e^{(\frac{N}{2}-1)\psi}u)$ for every smooth function $u.$
Since the Poincar\'e metric is conformal to the Euclidean metric with $\psi (x)= \ln \left(\frac{1-|x|^2}{2}\right)$, we can transform \eqref{Pa} in the Euclidean space as follows:
 Let $u$ be a solution to \eqref{Pa}. Define $\varphi:= e^{-(\frac{N}{2}-1)\psi}=\left(\frac{2}{1-|x|^{2}}\right)^{\frac{N-2}{2}}$. Then, the function $v :=\varphi u$ satisfies
 \begin{align}\label{con-ecl-eq}
 -\Delta v- c(x)v \;=\;a(x)|v|^{2^{\star}-2}v, \ \ \ v \in H^{1}_{0}(B(0,1)),
 \end{align} 
 where $c(x) = \frac{4\lambda-N(N-2)}{\left(1-|x|^{2}\right)^{2}}$. Here, $H^{1}_{0}(B(0,1))$ refers to the Sobolev space consisting of functions defined on  $B(0,1)$ that vanish on the boundary $\partial B(0,1).$  It is important to note that $c(x) > 0$ in $B(0,1)$ provided that $\lambda > \frac{N(N-2)}{4}.$

\medskip
\noindent
Let us denote the energy functional corresponding to \eqref{con-ecl-eq} by
\begin{equation*}
\mathcal J_a(v)=\frac{1}{2} \int_{B(0,1)}\left[|\nabla v|^2- c(x) v^2\right] \mathrm{~d}x-\frac{1}{2^*} \int_{B(0,1)} a(x)|v|^{2^*} \mathrm{~d}x.
\end{equation*}
For any $u \in  H^1(\mathbb B^N)$, define $\tilde{u}:=\left(\frac{2}{1-|x|^2}\right)^{\frac{N-2}{2}} u$. Then $\mathcal I_a(u)=\mathcal J_a(\tilde{u})$. In addition, when $\tilde{v}$ is defined similarly, $\left\langle \mathcal I_a^{\prime}(u), v\right\rangle=\left\langle \mathcal J_a^{\prime}(\tilde{u}), \tilde{v}\right\rangle$.  Also, let $\tilde{\mathcal{N}}$ denote the Nehari Manifold corresponding to $\mathcal J_a$, i.e., 
\be
\mathcal{\tilde{N}}:=\left\{u\in  H_0^1(B(0,1))\setminus\{0\}\,\,\text{ such that }\,\,\langle\,\mathcal J_a'(u),u\,\rangle=0  \right\}. \label{NeManEuc}
\ee
As a result, we get 
\begin{equation}
\inf_{\mathcal{N}} \mathcal I_a= \inf _{\mathcal{\tilde{N}}} \mathcal J_a. \label{con_cha}
\end{equation}
On $\mathcal{\tilde{N}}$, $\mathcal{J}_a$ takes the form
\begin{equation*}
\mathcal J_a(v)=\frac{1}{N}\left(\int_{B(0,1)}\left[|\nabla v|^2- c(x) v^2\right] \mathrm{~d}x\right)^{\frac{N}{2}}\left(\int_{B(0,1)} a(x) |v|^{2^*} \mathrm{~d}x \right)^{\frac{(2-N)}{2}}.
\end{equation*}
We consider the extension of this functional to the whole $H_0^1(B(0,1))$, we call it  $\tilde{\mathcal J}_a$.
 \begin{theorem} \label{Exi-res1}
     Let $N \geq 4.$ Assume that \eqref{a_maincond} is satisfied and the assumption c) holds. Then at least one positive solution of \eqref{Pa} exists. 
 \end{theorem}
 \begin{proof}
 Note that condition $(c)$ implies that there exists some $x_0\in\bn$ such that $a(x_0)>1$. If not, then $a(x)\leq 1 \;\;\forall x\in\bn$, which in turn would imply $\max a^\frac{2-N}{2}\geq 1$ and thus $D\leq \max a^\frac{2-N}{2} D\leq D_\la$, which contradicts the fact that $D_\la<D$. Therefore, $\max a>1$. Let $ y \in \mathbb{B}^N$ such that $a(y)=\max (a)$. Now observe that, to prove this theorem, it suffices to show that, $\tilde{\mathcal J}_a (\mathcal{\tilde V} )<a(y)^{(2-N) / 2} D$ for some $\mathcal{\tilde V} \in H_0^1(B(0,1))$. Indeed, for this $\mathcal{\tilde V}$, we can find $t_\mathcal{\tilde V}$ such that $t _ \mathcal{\tilde V} \mathcal{\tilde V} \in \mathcal{\tilde{N}}$, then performing an easy computation yields
     $$\mathcal J_a(t_\mathcal{\tilde V} \mathcal{\tilde V})= \tilde{\mathcal J}_a (\mathcal{\tilde V} ).$$
     Consequently, using the above equality and \eqref{con_cha}, we obtain
     $$\inf_{\mathcal{N}} \mathcal I_a=\inf _{\mathcal{\tilde{N}}} \mathcal J_a \leq \mathcal J_a(t_\mathcal{\tilde V} \mathcal{\tilde V}) = \tilde{\mathcal J}_a (\mathcal{\tilde V} )< a(y)^{(2-N) / 2} D.$$
 We can conclude the proof by combining the above estimate with Proposition~\ref{Pro_N_inf}. Therefore, we now aim to prove $\tilde{\mathcal J}_a (\mathcal{\tilde V} )<a(y)^{(2-N) / 2} D$ for some $\mathcal{\tilde V} \in H_0^1(B(0,1))$.
 
 \medskip 
 \noindent
 To this end, define
$$\mathcal{\tilde V}(x) := V^{\varepsilon, y}(x) \phi(x-y),$$
where $V$ is the unique positive solution of \eqref{Talenti} and $\phi \in C_c^{\infty}\left(\mathbb{R}^N\right)$ such that $0 \leq \phi \leq 1,\; \phi(z)=1$ for $|z|<r$ and $\phi(x)=0$ for $|z|>R$, and $R$ is chosen such that $B(y,R) \subseteq B(0,1).$ Set, $c(N):=[N(N-2)]^\frac{N-2}{4}$, $B:= B(0,1)$ and we estimate each term of $\tilde{\mathcal J}_a (\mathcal{\tilde V} )$ for the above chosen $\tilde{\mathcal{V}}$ as follows:
\begin{align*}
    \int_{B}|\nabla \mathcal{\tilde V}(x)|^2 \mathrm{~d}x &=  (c(N))^2  \int_{B} \left|\nabla\left(\[\frac{\varepsilon}{\varepsilon^2 +|x-y|^2}\]^\frac{N-2}{2} \phi(x-y)\right)\right|^2 \;{\rm d}x\\
       &= (c(N))^2  \int_{B} \left|\nabla\left(\[\frac{\varepsilon}{\varepsilon^2 +|z|^2}\]^\frac{N-2}{2} \phi(z) \right)\right|^2\mathrm{~d}z\\
       &= S^{\frac{N}{2}} + O(\varepsilon^{N-2})\\
       &= N D + O(\varepsilon^{N-2}).
    \end{align*}
  Next, if we define $\tilde \la=\la-\frac{N(N-2)}{4}$ then
  \begin{align*}
    \int_{B} c(x) |{\mathcal{\tilde V}(x)}|^2\mathrm{~d}x  &= \tilde\lambda (c(N))^2
\int_{B} \left(\frac{2}{1-|x|^2}\right)^2 \left|\[\frac{\varepsilon}{\varepsilon^2 +|x-y|^2}\]^\frac{N-2}{2} \phi(x-y)\right|^2\;{\rm d}x\\
& \geq \tilde\lambda (c(N))^2 C\int_{B} \left|\[\frac{\varepsilon}{\varepsilon^2 +|z|^2}\]^\frac{N-2}{2} \phi(z)\right|^2\mathrm{~d}z \\
&\geq \begin{cases}
C\tilde\la\eps^2+O(\eps^{N-2})   \qquad  \text{ if }\,\, N>4,\\
C\tilde\la\eps^2|\ln \eps|+O(\eps^{2}) \quad\text{ if }\,\, N=4.\\
\end{cases}
 \end{align*}  
Finally, since $y$ is the point of maximum, we have
\begin{align*}
\int_{B} a(x) |{\mathcal{\tilde V}(x)}|^{2^*} \mathrm{~d}x &= \underbrace{\int_{B} a(y) |{\mathcal{\tilde V}(x)}|^{2^*} \mathrm{~d}x}_{(I)} + \underbrace{\int_{B} [a(x)-a(y)-\langle\nabla a(y) \mid x-y\rangle]({\mathcal{\tilde V}(x)})^{2^*} \mathrm{~d}x}_{(II)}.
\end{align*}
First, we consider $(I)$
\begin{align*}
    (I):= \int_{B} a(y) |{\mathcal{\tilde V}(x)}|^{2^*} \mathrm{~d}x  &= (c(N))^{2^*}a(y) \int_{B(y,R)} \left|\[\frac{\varepsilon}{\varepsilon^2 +|x-y|^2}\]^\frac{N-2}{2} \phi(x-y)\right|^{2^*} \mathrm{~d}x\\
    & =(c(N))^{2^*}a(y)\int_{B} \left|\[\frac{\varepsilon}{\varepsilon^2 +|z|^2}\]^\frac{N-2}{2} \phi(z)\right|^{2^*}\mathrm{~d}z\\
    &= a(y)N D+O(\eps^N).
\end{align*}
Consider the integral $(II)$. Since $a$ is twice differentiable, bounded and asymptotically tends to $1$, there exists $C>0$ such that
$$|a(x)-a(y)-\langle\nabla a(y) \mid x-y\rangle|\leq C\frac{|x-y|^2}{1+|x-y|^2}.$$
Therefore, 
\begin{align*}
    |(II)|&\leq C \int_{B} \frac{|x-y|^2}{1+|x-y|^2}({\mathcal{\tilde V}(x)})^{2^*} \mathrm{~d}x
    = (c(N))^{2^*}\int_{B(0,R)} \frac{|z|^2}{1+|z|^2} \left|\[\frac{\varepsilon}{\varepsilon^2 +|z|^2}\]^\frac{N-2}{2} \phi(z)\right|^{2^*}\mathrm{~d}z\\
    & = (c(N))^{2^*}\int_{B(0,\frac{R}{\eps})} \frac{|\eps z|^2}{1+|\eps z|^2} \left| \left[\frac{1}{1+|z|^2}\right]^\frac{N-2}{2} \phi(\eps z)\right|^{2^*}\mathrm{~d}z
     \leq C \varepsilon^2 |V|_{L^{2^*}(\Rn)}^{2^*}= C \eps^2. 
    \end{align*}
 Thus   $$\int_{B} a(x)|{\mathcal{\tilde V}(x)}|^{2^*} \mathrm{~d}x=a(y) N D+O\left(\varepsilon^2\right).$$
    Now combining all the above estimates, for $N\geq 5$ we have 
    \begin{align*}
        \tilde{\mathcal J}_a (\mathcal{\tilde V})\leq& \frac{1}{N}\frac{\[N D + O(\varepsilon^{N-2})- C\tilde\la\varepsilon^{2}\]^\frac{N}{2}}{\[a(y) N D+O\left(\varepsilon^2\right)\]^\frac{N-2}{2}} 
         \approx a(y)^\frac{2-N}{2} D + O(\varepsilon^{N-2})- C\tilde\la\varepsilon^{2} <a(y)^\frac{2-N}{2} D,
    \end{align*}
    for $\eps>0$ small enough. \\
    For $N=4$
    \begin{align*}
        \tilde{\mathcal J}_a (\mathcal{\tilde V})\leq& \frac{1}{N}\frac{\[N D + O(\varepsilon^{2})- C\tilde\la\varepsilon^{2}|\ln\eps|\]^2}{\[a(y) N D+O\left(\varepsilon^2\right)\]} 
 \leq a(y)^{-1} D -C\tilde\la\varepsilon^{2}|\ln\eps| + O(\varepsilon^{2})
        <a(y)^{-1} D,
    \end{align*}
    for $\eps>0$ small enough. 
 \end{proof}
\medskip
\subsection{Higher Index Solutions: } In this subsection, $\inf_{u\in\mathcal N}\mathcal I_a\,\geq D_{\lambda}$ is the scenario under examination. Previously, using Proposition \ref{Pro_N_inf}, we discussed the existence of solutions when the opposite inequality was considered in the preceding subsection. To find a solution in this case, we utilize the min-max level and require the Lemma \ref{enerEstLem}. Before moving forward, note that on $\mathcal{N}$,  $\mathcal I_{a}$ can can be expressed as:
\begin{equation*}
\mathcal I_a(u)=\frac{1}{N}\left(\int_{ \mathbb B^N}\Big[|\nabla_{\mathbb B^N}u|^2-\lambda u^2\Big]\dvg\right)^{N / 2}\left(\int_{\mathbb B^N}a(x)|u|^{2^*}\dvg\right)^{(2-N) / 2},
\end{equation*}
which is beneficial because it is homogeneous of order zero. Thus, it can be useful to extend this homogeneous functional of order zero to the entire $H^{1}(\mathbb{B}^N);$ we refer to this extension as $\tilde{\mathcal I}_a$.\\
Firstly, we claim the existence of a positive constant $R_{1}$ for which the following holds:
		$$\tilde{\mathcal I}_a(\mathcal{U}(\tau_{-y}(\cdot))) \leq D_\lambda\quad \text{for every} \quad y \in \mathbb{B}^{N} \quad \text{with} \quad d(y,0) \geq R_{1}.$$
		To justify this assertion, we utilize \eqref{a_secondcond} and apply Fatou's lemma, yielding:
  \begin{equation}
			\begin{aligned}
				\int_{\mathbb{B}^{N}}|\mathcal{U}(x)|^{2^*} \mathrm{~d} V_{\mathbb{B}^{N}}(x) &
				 \leq \liminf\limits_{d(y,0)\rightarrow \infty} \int_{\mathbb{B}^{N}}  a(\tau_{y}(x))|\mathcal{U}(x)|^{2^*} \mathrm{~d} V_{\mathbb{B}^{N}}(x). \label{4.ut}
			\end{aligned} 
		\end{equation}
Therefore, we conclude that
  \begin{equation}
		\begin{aligned}
			\lim _{d(y,0) \rightarrow \infty} \tilde{\mathcal I}_a(\mathcal{U}(\tau_{-y}(\cdot))) & = \frac{1}{N}  \frac{\|\mathcal{U}\|_{\lambda}^{N}}{\liminf\limits_{d(y,0)\rightarrow \infty}\left(\int_{\mathbb{B}^{N}} a(\tau_{y}(x))|\mathcal{U}(x)|^{2^*} \mathrm{~d} V_{\mathbb{B}^{N}}\right)^{\frac{N}{2^*}}}
			 \leq \frac{1}{N}\frac{\|\mathcal{U}\|_{\lambda}^{N}}{\left(\int_{\mathbb{B}^{N}}|\mathcal{U}(x)|^{2^*} \mathrm{~d} V_{\mathbb{B}^{N}}\right)^{\frac{N}{2^*}}}\\
			& = \mathcal I_1(\mathcal{U})=D_\lambda, \\
		\end{aligned}\label{upbd}
  \end{equation}
		where we have used \eqref{4.ut} to get the last inequality. Hence the claim follows.

 \begin{lemma}\label{enerEstLem}
     Let $a(.)$ satisfy \eqref{a_maincond}, and there exist some positive constants $C$ and $\delta > 0$ such that 
     \begin{equation}\tag{\textbf{A1}}\label{a_secondcond}
     a(x) \geqslant 1 -\operatorname{C \, exp}(-\delta \, d(x,0))  \quad \forall \ d(x, 0) \rightarrow \infty,
     \end{equation}
and $\mathcal U$ be the unique positive solution of \eqref{E}. Then, there exists a large number $R_{0}$, such that for any $R \geq R_{0},$ and for $d(x_{1},0) \geq R^{\alpha},\; d(x_{2},0)\geq R^{\alpha},\; R^{\alpha^{\prime}} \leq d\left(x_{1},x_{2}\right) \leq  R^{\alpha^{\prime}-\alpha} \min \left\{d(x_{1},0), \;d(x_{2},0)\right\}$, where $\alpha>\alpha^{\prime}>1$,  the following inequality holds
	\begin{equation}
		\tilde{\mathcal I}_a\left(t u_{1}+(1-t) u_{2}\right)<2 D_\lambda, \label{ener_est}
	\end{equation}
	where $0 \leq t \leq 1$ and $ u_{i}=\mathcal U\left(\tau_{-x_{i}}(\cdot)\right)$ for $i=1,2$.
 \end{lemma}

 \begin{proof}
The proof proceeds along the same lines as in \cite[Lemma 4.2]{GGS2}. We shall sketch the proof here for the sake of completeness.      
  Set, $A:=\|\mathcal U\|_{\lambda}^{2}=\int_{\mathbb{B}^{N}} \mathcal U(x)^{2^*} \mathrm{~d} V_{\mathbb{B}^{N}}.$ Thus, $D_\lambda= \frac{1}{N} A$.\\
First, we shall prove \eqref{ener_est} for $t=\frac{1}{2}$. For $u:= \frac{u_1+u_2}{2},$ we have 
\begin{equation*}
\tilde{\mathcal I}_a(u)=\frac{1}{N}\|u\|_\lambda^N\underbrace{\left(\int_{\mathbb B^N}a(x)|u|^{2^*}\dvg(x)\right)^{(2-N) / 2}}_{I},
\end{equation*}
where
\begin{align}\label{est1}
     \|u\|_\lambda^2=\frac{1}{4}\[\left\|u_{1}\right\|_{\lambda}^{2}+\left\|u_{2}\right\|_{\lambda}^{2}+2\left\langle u_{1}, u_{2}\right\rangle_{\lambda}\]&= \frac{1}{2}\[A+ \left\langle u_{1}, u_{2}\right\rangle_{\lambda}\]\\
     &=\frac{1}{2}\[N D_\lambda+ \int_{\mathbb{B}^{N}} u_{1}^{2^*-1} u_{2} \mathrm{~d} V_{\mathbb{B}^{N}}\],
\end{align}
where we have used the fact that $u_{1}, u_{2}$ also solve \eqref{E}. Thereafter,
   \begin{align}
       I^\frac{2}{2-N}&:= \int_{\mathbb B^N}a(x)|u|^{2^*}\dvg(x)=\underbrace{\int_{\mathbb{B}^{N}}\left(u_{1}+u_{2}\right)^{2^*} \mathrm{~d} V_{\mathbb{B}^{N}}}_{I_1}-\underbrace{\int_{\mathbb{B}^{N}}(1-a(x))\left(u_{1}+u_{2}\right)^{2^*} \mathrm{~d} V_{\mathbb{B}^{N}}}_{I_2}. \notag
   \end{align}
Note that for any non-negative real numbers $a$ and $b$, the following holds
	\begin{equation*}
		(a+b)^{2^*} \geq a^{2^*}+b^{2^*}+(2^*-1)\left(a^{2^*-1} b+a b^{2^*-1}\right). 
	\end{equation*}
 Therefore, 
 \begin{align}  \label{est2}
     I_1= \int_{\mathbb{B}^{N}}\left(u_{1}+u_{2}\right)^{2^*} \mathrm{~d}V_{\mathbb{B}^{N}} \geq 2 (A+(2^*-1)\left\langle u_{1}, u_{2}\right\rangle_{\lambda})= 2(N D_\lambda+ (2^*-1)\int_{\mathbb{B}^{N}} u_{1}^{2^*-1} u_{2}\mathrm{~d} V_{\mathbb{B}^{N}}).
 \end{align}
 Now following the arguments as in \cite[pg. 15-17]{GGS1}, we obtain
 \begin{align}   
     I_2 := \int_{\mathbb{B}^{N}}(1-a)\left(u_{1}+u_{2}\right)^{2^*} \mathrm{~d} V_{\mathbb{B}^{N}} \leq o(1)\left(\int_{\mathbb{B}^{N}} u_{1}^{2^*-1} u_{2} \mathrm{~d} V_{\mathbb{B}^{N}}\right) \label{est3}
 \end{align}
 for large $R>0$ and $\delta < 2^* c(N, \lambda)$. In these calculations, we utilized the hypothesis $d\left(x_{1},x_{2}\right) \leq R^{\alpha^{\prime}-\alpha}d(x_i,0)$ for $i=1,2.$\\
Combining the estimates \eqref{est1}-\eqref{est3}, we obtain
 \begin{align*}
\tilde{\mathcal I}_a(u)&
\leq\frac{1}{N}\left(\frac{1}{2}\[N D_\lambda+ \int_{\mathbb{B}^{N}} u_{1}^{2^*-1} u_{2}\mathrm{~d} V_{\mathbb{B}^{N}}\] \right)^\frac{N}{2}\times\\
&\qquad\quad\left(2 N D_\lambda+2(2^*-1)\int_{\mathbb{B}^{N}} u_{1}^{2^*-1} u_{2}\mathrm{~d} V_{\mathbb{B}^{N}}-o(1)\int_{\mathbb{B}^{N}} u_{1}^{2^*-1} u_{2} \mathrm{~d} V_{\mathbb{B}^{N}}\right)^\frac{2-N}{2}\\
&\leq \frac{1}{N} \left(\frac{N D_{\lambda}}{2}\right)^\frac{N}{2}\left(2N D_\lambda\right)^\frac{2-N}{2}\left(1+\frac{\int_{\mathbb{B}^{N}} u_{1}^{2^*-1} u_{2}\mathrm{~d} V_{\mathbb{B}^{N}}}{N D_\lambda}\right)^\frac{N}{2}\times\\
&\qquad\qquad\qquad\qquad\left(1+\frac{(2^*-1-o(1))\int_{\mathbb{B}^{N}} u_{1}^{2^*-1} u_{2}\mathrm{~d} V_{\mathbb{B}^{N}}}{ND_\lambda}\right)^\frac{2-N}{2}\\
& \leq 2 D_\lambda \bigg(1+\big(\frac{D_\lambda ^{-1}}{2}+o(1)\big)\int_{\mathbb{B}^{N}} u_{1}^{2^*-1} u_{2}\mathrm{~d} V_{\mathbb{B}^{N}}\bigg)\bigg(1-\big(\frac{N+2}{2N}D_\lambda ^{-1}+o(1)\big)\int_{\mathbb{B}^{N}} u_{1}^{2^*-1} u_{2}\mathrm{~d} V_{\mathbb{B}^{N}}\bigg)\\
& \leq 2 D_\lambda - \bigg(\frac{2}{N}+o(1)\bigg)\int_{\mathbb{B}^{N}} u_{1}^{2^*-1} u_{2}\mathrm{~d} V_{\mathbb{B}^{N}}<2D_\lambda, 
\end{align*}
 for large $R$, using \cite[Lemma 4.1]{AX14618}.  \\
Now for $t\in[0,1]\setminus \frac{1}{2}$, set
				$v_{1}=t u_{1}$ and  $v_{2}=(1-t) u_{2}$. Then, a straightforward computation (see \cite[Lemma 4.2]{GGS1} for further details) implies
				\begin{align*}
\tilde{\mathcal{I}}_a \left(v_{1}+v_{2}\right) &=\frac{1}{N}\left(\left(t^{2}+(1-t)^{2}\right) A+2 t(1-t)\left\langle u_{1}, u_{2}\right\rangle_{\la}\right)^\frac{N}{2}\times\\
&\qquad\left(\int_{\mathbb{B}^{N}}\left|v_{1}+v_{2}\right|^{2^*} \mathrm{~d} V_{\mathbb{B}^{N}}-\int_{\mathbb{B}^{N}}(1-a)\left|v_{1}+v_{2}\right|^{2^*} \mathrm{~d} V_{\mathbb{B}^{N}}\right)^\frac{2-N}{2}\\
&\leq \frac{1}{N} \left(\left(t^{2}+(1-t)^{2}\right) A+2 t(1-t)\left\langle u_{1}, u_{2}\right\rangle_{\la}\right)^\frac{N}{2}\times\\
&\qquad\left[\left(t^{2^*}+(1-t)^{2^*}\right) A+(2^*-1)\left(t^{2^*-1}(1-t)+t(1-t)^{2^*-1}-o(1)\right)\left\langle u_{1}, u_{2}\right\rangle_{\lambda}\right]^{\frac{2-N}{2}}\\
& \leq \frac{1}{N} A \frac{\left(t^{2}+(1-t)^{2}\right)^\frac{N}{2}}{\left(t^{2^*}+(1-t)^{2^*}\right)^{\frac{N-2}{2}}} + o(1) < 2 D_{\lambda},
\end{align*}
for large $R$ and using the fact that 
$\frac{\left(t^{2}+(1-t)^{2}\right)^\frac{N}{2}}{\left(t^{2^*}+(1-t)^{2^*}\right)^{\frac{N-2}{2}}}<2$ for $t \neq 1/2.$ 
 \end{proof}
Now, select $R_0$ as specified in Lemma \ref{enerEstLem}, and then choose $R_2$ such that $R_{2}>\max \left\{R_{0}, R_{1}\right\}$. Furthermore, set $\mathcal U\left(\tau_{-y}(\cdot)\right) = \mathcal{U}_{y},$ and  define:
$$
\begin{aligned}
& \Gamma:=\left\{\gamma \in \mathcal{C}\left(B(0,R_2)^\complement, \mathcal{N}\right) \text { s.t. } \exists\,  \tilde R>R^\prime>0 \mid \gamma(y)=Q\left(\mathcal{U}_{-y}\right) \forall y \text{ such that } d(y,0)> \tilde R\right. \\
&\hspace{3cm} \left.\text { and } \gamma(y)=Q\left(\mathcal{U}_{-y}\right) \, \forall\, y \text{ such that } d(y,0)< R^\prime \right\},
\end{aligned}
$$
where $Q$ is the central projection on $\mathcal{N}$, i.e., $u \mapsto t_{u}u$ where $t_{u}$ is defined as in \eqref{scaNeh}.
Then define
$$
d_{\lambda}:=\inf_{\gamma \in \Gamma}  \max_{y \in B(0,R_2)^\complement} \mathcal I_a(\gamma(y)) .
$$
Fix $R_4$ sufficiently large such that $R_4, R_{4}-R_{4}^{\frac{\alpha^{\prime}}{\alpha}}> R_2^\alpha>R_2$. Denote $R_3=R_{4}-R_{4}^{\frac{\alpha^{\prime}}{\alpha}}$.
And for this fixed $R_4$, choose 
		\begin{equation*}
			x_{2}=\left(0,0, \ldots, \tanh{\left(\frac{R_{4}-R_{4}^{\frac{\alpha^{\prime}}{\alpha}}}{2}\right)}\right).
		\end{equation*} 

Using the previous Lemma, we can attain that $d_{\lambda}<2 D_{\lambda}$. Specifically, we define a path $\gamma \in \Gamma$ such that $\gamma(y):=Q\left(\mathcal{U}_{-y}\right)$ if $d(y,0)<R_3$ or $d(y,0)>R_4$. For $\gamma$ restricted to the annulus with an inner radius of $R_3$ and an outer radius of $R_4$, we parameterize the image by $Q$ of the segment joining $\mathcal U_{-x_1}$ and $\mathcal U_{-x_2}$ for $x_1 \in \partial B(0,R_4)$ and $x_2 \in \partial B(0,R_3)$, i.e., $\gamma\left(t x_1+(1-t)x_2\right)= t_{t \mathcal U_{-x_1}+(1-t)\mathcal U_{-x_2}}\left(t \mathcal U_{-x_1}+(1-t)\mathcal U_{-x_2}\right)$. Using \eqref{upbd},
$$
\max _{y \in \left(B(0,R_3)\setminus B(0,R_2)\right) \cup B(0,R_4)^\complement} \mathcal{I}_a(\gamma(y))<2 D_{\lambda} .
$$

Thus, by the choice of $x_1$ and $x_2$, and using Lemma \ref{enerEstLem}, along with the definition of $d_\lambda$, we can conclude that $d_\lambda < 2D_\lambda$.

\medskip

\begin{theorem}
    Let $N \geq 4.$ Suppose $a(x)> 1$ for all $x \in \bn$, and we are in either case (a) or (c). Then \eqref{Pa} possesses at least one positive solution.
\end{theorem}

\begin{proof}We can extend the definition of $\Gamma$ to include all of $\bn$. Within the ball $B(0,R_2)$, the inequality $d_{\lambda}< 2D_{\lambda}$ can be immediately derived from $a(x)>1$. This leads to two possible scenarios:

\medskip 

Under the condition that $d_{\lambda}=D_{\lambda}$, there exists a sequence $\left(u_{n}\right) \subset \mathcal{D}_{+}^{1,2}\left(\bn\right)$ such that $\mathcal I_{a}\left(u_{n}\right) \rightarrow D_{\lambda}$ and

$$
G\left(u_{n}\right):=\int_{\bn}\left|\nabla u_{n}\right|^{2} \dvg \cdot\left(\int_{\bn}\left|\nabla \mathcal{U}\right|^{2} \dvg \right)^{-1}=c<1 .
$$
For every path $\gamma \in \Gamma,\lim _{d(y,0) \rightarrow 0} G(\gamma(y))=\[\frac{\|\mathcal{U}\|^2}{\int_{\bn}a(z)\mathcal{|U|}^{2^*}\dvg(z)}\]^{\frac{2}{2^*-2}}<1$ and $\lim _{d(y,0)+\infty} G(\gamma(y))=1$. According to Ekeland's variational principle, there is a Palais-Smale sequence $\left(v_{n}\right) \subset \mathcal{D}_{+}^{1,2}\left(\bn\right)$ such that $G\left(v_{n}\right) \rightarrow c$ and $\mathcal{I}_{a}\left(v_{n}\right) \rightarrow D_{\lambda}$. Therefore, by utilizing Theorem \ref{ps_decom}, the sequence $\left(v_{n}\right)$ possesses a subsequence that converges weakly to a nonzero solution of \eqref{Pa}.

If $d_{\lambda}>D_{\lambda}$, according to the Mountain Pass lemma, there exists a Palais-Smale sequence for $\mathcal I_{a}$ at level $d_{\lambda}$. Additionally, since we have established $d_{\lambda}<2 D_{\lambda}$, and under the conditions of case $(a)$ or $(c)$, $d_{\lambda}$ falls within the safe energy range.
\end{proof}

\section{Existence of solutions: The case without Palais-Smale }
\label{WithoutPS}
The situation where the Palais-Smale criterion is not satisfied is the focus of this section. Here, we will address the remaining scenario described below:
\begin{equation}\label{infty-case}
D_\lambda<\max (a)^{(2-N) / 2} D<2 D_\lambda \quad \text { and } \quad D_\lambda<d_\lambda<2 D_\lambda .
\end{equation}

\medskip

Indeed, the same result as in the preceding section applies if $d_\lambda<\max (a)^{(2-N) / 2} D_0$.  This is why the problem of reversed inequality is interesting. It is delicate because the usual variational arguments will not be sufficient to address this scenario. We shall follow the approach of A.~Bahri's \it critical point at infinity \rm to establish solutions to the problem. The critical exponent problem with Hardy potential in the entire Euclidean space was studied in the seminal paper \cite{Sm}, which proved a solution in dimension $N=4.$ The dimension restriction is quite delicate because few estimates are limited to that dimension. We will specifically utilise Smets' concepts in our context.
Nonetheless, the problem in the hyperbolic space will be reduced to an equation in the Euclidean ball with a singularity at the ball's boundary via a conformal change of metric. Because of Theorem~\ref{ps_decom}, we will avoid the boundary when conducting blow-up analysis. Local blow-up only occurs within the ball's interior of radius $2 - \sqrt{3}$ as demonstrated in the proof of \cite[Theorem 3.1]{BGGS}. Furthermore, unlike in the entire Euclidean space, we now have to deal with the bounded domain, hence projection-type arguments must be used. This adds a great deal of complexity to the subsequent estimates.

\medskip

Let $y_0 \in B(0,1)$ be such that $d_\lambda = a\left(y_0\right)^{(2-N)/2} D$, and define $A^*:= \left\{x \in B(0,1) : a(x) = a\left(y_0\right)\right\}$. Assume that  $a(y_0) \neq 1.$ Let $A_*$ denote a compact neighbourhood of $A^*$ within $B(0,1)$. Thus, for any $y \in B(0,1)$, there exists $\bar{y} \in A_*$ such that $\operatorname{dist}\left(y, A_*\right) = |y - \bar{y}|$. Furthermore, we impose the following assumption on the potential $a(x)$:
\begin{enumerate}[label=(\textbf{A\arabic*})]
\item  There exists $\theta $ such that $4 < \theta <N$ and for any $\tilde{y} \in A_*, |a(x) - a(\tilde{y})| = o\left(|x - \tilde{y}|^\theta\right)$ as $x \rightarrow \tilde{y}$. \label{A2}
\end{enumerate}

\medskip 
This additional assumption on the potential $a(x)$ is motivated by the work of (\cite{FP, AFP}), such an assumption is sometimes called the non-degeneracy condition on $a(x).$
 We will now prove the main theorem of this section, assuming $a(y_0) \neq 1$.

\begin{theorem}\label{theorem-critical-point-infinity}
 Let $N >6.$ Assume that \eqref{a_maincond}, \ref{A2} are satisfied and the assumption \eqref{infty-case} holds. Then \eqref{Pa} has at least one positive solution $u \in H^{1}(\bn)$ for all $\frac{N(N-2)}{4} < \lambda < \frac{N(N-2)}{4} + \frac{1}{N+4}$.
 \end{theorem}
 

\begin{remark}
{\rm 
The dimension restriction $N > 6$ appeared in the intermediate results while proving crucial estimates. Interestingly, only in dimension $N =4$ was the existence of the solution shown in \cite{Sm}. The analysis carried out in that dimension will not be extended, specifically because of an integrability problem resulting from the Hardy term at the origin. Similarly, in our case, a singularity appears at the Euclidean ball boundary following the conformal change of metric, but our dimension restriction avoids the integrability problems.
}
\end{remark}


\begin{remark}
{\rm The limitation on the upper bound of $\lambda$ is another shortcoming of Theorem~\ref{theorem-critical-point-infinity}. 
This limitation results from the sharp constant in the \it spectral inequality, \rm which O.~Rey \cite[Appendix~D]{ReyO} demonstrated. Although we would prefer to continue until $\frac{(N-1)^2}{4}$, we believe that this restriction on $\lambda$ is reasonable given that, in the process of proving our theorem, we will be interested in the case where $d_\lambda= a\left(y_0\right)^{(2-N) / 2} D$. As a result, we will conduct the analysis close to $d_\lambda$, or close to the localised Aubin-Talenti Bubble.  So, when $\lambda$ is close to $\frac{N(N-2)}{4},$ the energy of the associated functional will be quite close to $D$.
}
\end{remark}

To this end, we shall first recall some of the well-known results concerning projection operators, the Green's function, its regular part, and the strategy of the proof.


\medskip
\subsection{Projections and Setting up the Problem}
Let $P$ denote the projection from $H^1(B(0,1))$ onto $H_0^1(B(0,1))$, i.e.,
$u=P f$ is the solution of :
\begin{equation}
 	\tag{$\mathcal{P}$}\label{ProjOp}
  \left\{\begin{aligned}
     \Delta u=\Delta f & \text { on } B(0,1), \\ u=0 & \text { on } \partial B(0,1).
    \end{aligned}
    \right.
    \end{equation}
For $V^{\varepsilon, y}$ as defined in \eqref{Talenti} and $P$ as defined in \eqref{ProjOp}, we have the following lemma:

\begin{lemma}{\cite[Proposition 5.2]{BahriB}}
Let $N \geq 4.$ For all $\mu>0$, there exists $\nu>0$ such that if $u \in H_0^1\left(B(0,1)\right)$ satisfies $$\left\|u-cP V^{\varepsilon, y}\right\|_{H_0^1\left(B(0,1)\right)}<\nu$$ for some $0<\varepsilon<\nu, c>\mu$, and $y \in B(0,1)$, then there exist unique $y(u),\varepsilon(u), c(u), \delta(u)>0$, and $v(u) \in H_0^1\left(B(0,1)\right)$ such that
\begin{equation*}
u=c(u)P V^{\varepsilon(u), y(u)}+\delta(u) v(u)
\end{equation*}
where $v$ satisfies the orthogonality condition
  \begin{equation}
\tag(*{\textasteriskcentered}\label{OrthCondn}\left\{\begin{array}{l}
\left\langle PV^{\varepsilon(u), y(u)}, v\right\rangle=0, \\
\left\langle  \partial_{y_i} PV^{\varepsilon(u), y(u)}, v\right\rangle=\left\langle  \partial_{\varepsilon} PV^{\varepsilon(u), y(u)}, v\right\rangle=0, \quad 1 \leq i \leq N,
\end{array}\right.
\end{equation}
and $\|v\|_{H_{0}^{1}(B(0,1))}=1$.
\end{lemma}

\begin{remark}
    {\rm It is evident that any $w \in H_0^1(B(0,1))$ naturally extends by zero outside $B(0,1)$. We will refer to this extension of $w$ by $\tilde w$. Additionally, $\tilde w  \in H^1(\Rn)$.}
\end{remark}
\noindent
Furthermore, employing integration by parts, the definition of projection, and \eqref{OrthCondn}, we can infer
\begin{equation}
\tag*{(\textasteriskcentered\textasteriskcentered)}\label{OrthCondn2}\left\{\begin{array}{l}
\left\langle V^{\varepsilon, y}, \tilde{v}\right\rangle_{H^1(\Rn)}=0, \\
\left\langle  \partial_{y_i} V^{\varepsilon, y}, \tilde{v}\right\rangle_{H^1(\Rn)}=\left\langle  \partial_{\varepsilon} V^{\varepsilon, y}, \tilde{v}\right\rangle_{H^1(\Rn)}=0, \quad 1 \leq i \leq N.
\end{array}\right.
\end{equation}
Also, it is important to observe that by utilizing $\emph{(1)}$ of Estimate \ref{ErrorPro}, we obtain $P  V^{\varepsilon, y} \leq  V^{\varepsilon, y}.$
\begin{remark}
{\rm
The functions $c(.), \varepsilon(.), y(.), \delta(.)$ and $v(.)$ are smooth.  We will neglect to indicate the dependency of these coefficients on $u$ from now on,  whenever it won't confuse. Also, we shall denote by $d=\operatorname{dist}(y, \partial B(0,1)).$}
\end{remark}
\subsection{Green's Function} Consider $\varphi: B(0,1) \rightarrow \mathbb{R}$ defined by
\begin{equation*}
\varphi(x)=H(x, x),
\end{equation*}
where $H(x, y)=\frac{1}{|x-y|^{N-2}}-G(x, y)$ on $B(0,1) \times B(0,1)$ with $G$ satisfying
\begin{equation*}
\forall x \in B(0,1) \begin{cases}-\Delta G(x, \cdot)=\delta_x & \text { on } B(0,1), \\ G(x, \cdot)=0 & \text { on } \partial B(0,1) .\end{cases}
\end{equation*}
Here, $\delta_x$ denotes the Dirac mass at $x$. Consequently, $H$ represents the regular part of the Green's function and satisfies
\begin{equation}
\left\{\begin{array}{l}
\Delta_y H(x, y)=0 \, \text { in } B(0,1), \\
H(x, y)=\frac{1}{|x-y|^{N-2}} \,\text { on } \partial B(0,1).
\end{array}\right. \label{RegGr}
\end{equation}
\subsection{Critical Point at Infinity and the Strategy of the Proof}

\medskip 

\medskip
\noindent
For a given $\rho>0$, we define the $\rho$-neighborhood of the critical set at infinity as
\begin{equation*}
\mathcal{I}_\rho:=\left\{c PV^{\varepsilon, y}+\delta v \text { s.t. }\left|c-a\left(y_0\right)^{\frac{2-N}{4}}\right|<\rho,\;\varepsilon<\rho,\;\delta<\rho, \text { and } \operatorname{dist}\left(y, A^*\right)<\rho\right\},
\end{equation*}
where $\|v\|_{H_{0}^{1}(B(0,1))}=1$ and $v$ satisfies the orthogonality condition \eqref{OrthCondn}.

Our goal is to prove the existence of $\rho>0$ and a sequence $\left(\gamma_n\right)_n \subset \Gamma$ such that  
\begin{equation*}
\max _{y \in B(0,R_2)^\complement} \mathcal J_{a}\left(\gamma_n(y)\right) \rightarrow d_\lambda \quad \text { and } \quad \gamma_n\left(B(0,R_2)^\complement\right) \cap \mathcal{I}_\rho=\emptyset \quad \forall n \in \mathbb{N} .
\end{equation*}

The above goal is demonstrated in Proposition~\ref{main-prop}. After proving this deformation argument and further utilizing the Mountain Pass Lemma, Theorem~\ref{ps_decom}, and Lemma \ref{ProjCutoff}, proving Theorem~\ref{theorem-critical-point-infinity} will be straightforward.

The proof of Proposition~\ref{main-prop} proceeds as follows:  We will use a path $\gamma \in \Gamma$ and two consecutive intersection times of $\gamma$ with $\mathcal{I}_\rho$, say $t_{\mathcal{U}_{{-y}_{1}}}$ and $t_{\mathcal{U}_{{-y}_{2}}}$, where $t_{u}$ for $0 \neq u \in H^1(\bn)$ is as defined in \eqref{scaNeh}. We then create a deformation $\tilde{\gamma}$ of $\gamma$ inside $\left[t_{\mathcal{U}_{{-y}_{1}}}, t_{\mathcal{U}_{{-y}_{2}}}\right]$ such that it no longer enters into $\mathcal{I}_\rho$ without raising the energy level. This proof method is inspired by \cite{BE}. The first step in doing this is to, without raising the energy level, connect $\gamma\left(y_1\right)$ and $\gamma\left(y_2\right)$, respectively, to functions for which $\delta=0$ but $\varepsilon>\rho$. After that, the path is closed while travelling along $\{\delta=0\}$, a finite-dimensional manifold.


\medskip

To achieve our goal, we need to prove the subsequent lemmas. 

\begin{lemma}\label{ProjCutoff}
Let $N \geq 4$ and $P$ denoting the projection operator as defined in \eqref{ProjOp}, the following holds
$$
 \|PV^{\varepsilon,y}-\phi V^{\varepsilon,y}\| \rightarrow 0 \quad \mbox{as} \ \varepsilon \rightarrow 0,
$$
where  $\phi$ is defined in \eqref{seq-vn} and $ \|.\| := \|.\|_{H^1_0\left(B(0,1)\right)}.$ In particular, $PV^{\varepsilon,y}$ and $\phi V^{\varepsilon,y}$ are close in $H_0^1(B(0,1))$ norm for $\varepsilon$ small.

\end{lemma}

\begin{proof}
We write the difference in $H^1_0(B(0, 1))$ as follows 
\begin{align*}
    \|PV^{\varepsilon,y}-\phi V^{\varepsilon,y}\|=\|PV^{\varepsilon,y}- V^{\varepsilon,y}-(\phi-1)V^{\varepsilon,y}\|\leq \underbrace{\|PV^{\varepsilon,y}- V^{\varepsilon,y}\|}_{I_1}+\underbrace{\|(\phi-1)V^{\varepsilon,y}\|}_{I_2}.
\end{align*}

Let us denote $B(0,1)$ by $B$, and we can estimate $I_1$ as follows 

\begin{align}
    I_1^2 &=\int_{B} |\nabla(PV^{\varepsilon,y}- V^{\varepsilon,y})|^2 \;{\rm d}x = \int_{B}\left|\nabla P V^{\varepsilon,y}\right|^2\;{\rm d}x+\int_{B}\left|\nabla V^{\varepsilon,y}\right|^2\;{\rm d}x - 2\int_{B}\nabla P V^{\varepsilon,y}. \nabla V^{\varepsilon,y}\;{\rm d}x \notag\\
    &=\int_{B}\left|\nabla P V^{\varepsilon,y}\right|^2\;{\rm d}x+ \int_{B}\left|\nabla V^{\varepsilon,y}\right|^2 \;{\rm d}x- 2\int_{B}\nabla P V^{\varepsilon,y}. \left(\nabla \varphi^{\varepsilon,y}+\nabla  P V^{\varepsilon,y} \right)\;{\rm d}x\notag\\ 
    & = - \int_{B}\left|\nabla P V^{\varepsilon,y}\right|^2\;{\rm d}x+ \int_{B}\left|\nabla V^{\varepsilon,y}\right|^2\;{\rm d}x- 2\int_{B}\nabla P V^{\varepsilon,y}. \nabla \varphi^{\varepsilon,y}\;{\rm d}x\notag\\
    &=- 2\int_{B}\nabla P V^{\varepsilon,y}. \nabla \varphi^{\varepsilon,y}\;{\rm d}x+ O\left(\left(\frac{\varepsilon}{d}\right)^{N-2}\right) = O\left(\left(\frac{\varepsilon}{d}\right)^{N-2}\right) \notag \text{ (using Estimate \ref{proGrad} and \eqref{AubGrad}}).
\end{align}
In the last step, we used integration by parts in the term $\int_{B(0,1)}\nabla P V^{\varepsilon,y}. \nabla \varphi^{\varepsilon,y}$, considering that $P V^{\varepsilon,y}=0 $ on $\partial B(0,1)$ and $\varphi^{\varepsilon,y}$ is harmonic. 

Next, we turn to estimate $I_2.$  From \eqref{seq-vn}, it follows that $\phi \equiv 1$ in $B(y, r-(2- \sqrt{3}))$ 
with $r > 2- \sqrt{3}.$ Also, it follows that $B(0,1) \subset B(y, R+1)$ for $y \in B(0, R).$ Indeed, for $x \in B(0,1),$ we have $|x-y|< |x|+|y|<1+R$. Moreover, it is easy to see that $\phi \equiv 0$ in $B(y, R+1)^c$. Indeed, for $z\in B(y, R+1)^c$, it holds that $|z| \geq |z-y|-|y| > R+1-1=R.$\\
By renaming $R+1$ as $\tilde{R}$ and $r-(2- \sqrt{3})$ as $\tilde{r}$, and denoting $ B(y,\tilde{r})$ by $B_{\tilde{r}}$, $ B(y,\tilde{R})$ by $B_{\tilde{R}}$, we obtain
\begin{align*}
    \|(\phi-1)V^{\varepsilon,y}\|^2\notag &= \int_{B}\left|\nabla\left[(\phi-1)V^{\varepsilon,y}\right]\right|^2 \mathrm{~d}x\\
    &=\int_{B \setminus B_{\tilde{r}}} \left|\nabla (\phi-1)\right|^2 (V^{\varepsilon, y})^2\mathrm{~d}x+ \int_{B \setminus B_{\tilde{r}}} \left|\nabla V^{\varepsilon, y}\right|^2 (\phi-1)^{2}\mathrm{~d}x\\
    &\qquad+
    2\int_{B \setminus B_{\tilde{r}}} (\phi-1)V^{\varepsilon, y}\nabla (\phi-1) . \nabla V^{\varepsilon, y}\mathrm{~d}x\notag\\
    &\leq c \left[\int_{B_{\tilde{R}} \setminus B_{\tilde{r}}} \left|\nabla (\phi-1)\right|^2 (V^{\varepsilon, y})^2\mathrm{~d}x+ \int_{B \setminus B_{\tilde{r}}}\left|\nabla V^{\varepsilon, y}\right|^2 (\phi-1)^{2}\mathrm{~d}x \right.\notag\\
    &\qquad+\left.\int_{B_{\tilde{R}} \setminus B_{\tilde{r}}} (\phi-1)V^{\varepsilon, y}\nabla (\phi-1) . \nabla V^{\varepsilon, y}\mathrm{~d}x\right]\notag\\
     &\leq c(N) \left[\int_{B_{\tilde{R}}\setminus B_{\tilde{r}}}\frac{\varepsilon^{N-2}}{\left[\varepsilon^2+|x-y|^2\right]^{N-2}}\mathrm{~d}x + \int_{B \setminus B_{\tilde{r}}}\frac{\varepsilon^{N-2}|x-y|^2}{\left[\varepsilon^2+|x-y|^2 \right]^N}\mathrm{~d}x \right.\notag\\
      & \hspace{1.5cm} + \left. \left(\int_{B_{\tilde{R}} \setminus B_{\tilde{r}}} \left|V^{\varepsilon, y}\nabla(\phi-1)\right|^{2}\mathrm{~d}x\right)^{1/2}  \left(\int_{B_{\tilde{R}} \setminus B_{\tilde{r}}} \left|\nabla V^{\varepsilon, y}(\phi-1)\right|^{2} \mathrm{~d}x\right)^{1/2} \right]\notag\\
       &\leq c(N)\left[ \varepsilon^2 \int_{\Rn}V^2 \mathrm{~d}z + \int_{\frac{\tilde{r}}{\varepsilon}}^{\infty}\frac{|z|^2}{\left[1+|z|^2 \right]^N} \mathrm{~d}z \right.\notag\\
         &\qquad+\left.\left(\int_{B_{\tilde{R}} \setminus B_{\tilde{r}}} \left|V^{\varepsilon, y}\right|^{2} \mathrm{~d}x\right)^{1/2}  \left(\int_{B_{\tilde{R}} \setminus B_{\tilde{r}}} \left|\nabla V^{\varepsilon, y}\right|^{2} \mathrm{~d}x\right)^{1/2} \right]\notag\\
       &\leq c(N)\left[\varepsilon^2+\varepsilon^{N-2}+\varepsilon^{\frac{N}{2}}\right]= O(\varepsilon^2), \label{eqn1c}
\end{align*}
for $N \geq 4.$ Therefore, 
$PV^{\varepsilon,y}-\phi V^{\varepsilon,y} \rightarrow 0$ as $\varepsilon \rightarrow 0$ in $H_{0}^1(B(0,1))$.
\end{proof}

In the subsequent computations, we shall perform the estimations using the energy functional after the conformal change. Specifically, we will work with the functional $\tilde{\mathcal J}_a$ associated with the Aubin-Talenti bubbles. Thus, for some estimations, we may set the multiplicative factor $c$ to be $1$.\\
In the subsequent computations, we refer to $B(0,1)$ as $B$. The initial lemma provides an energy control for $\tilde{\gamma}$ on $\{\delta=0\}$.
\begin{lemma} \label{energyMain}
For $N \geq 4$, there exists $0 < \varepsilon_1^{+} < \nu$ such that for all $\varepsilon < \varepsilon_1^{+}$, there exists $\tau > 0$ such that if  $\operatorname{dist}(y, A^*) < \tau$, then $\tilde{\mathcal{J}}_a(P V^{\varepsilon, y}) < d_\lambda = a(y_0)^{(2-N)/2} D$.
\end{lemma}
\begin{proof}
\begin{equation*}
\tilde{\mathcal J}_a(P  V^{\varepsilon, y})=\frac{1}{N}\left(\int_{B}\left[|\nabla P  V^{\varepsilon, y}|^2- c(x) (P  V^{\varepsilon, y})^2\right]\mathrm{~d}x\right)^{\frac{N}{2}}\left(\int_{B} a(x) |P  V^{\varepsilon, y}|^{2^*}\mathrm{~d}x \right)^{\frac{(2-N)}{2}}.
\end{equation*}   
Let's proceed to estimate each term in the equation above. First, we observe that, applying the Estimate \ref{proGrad}, we obtain,
\begin{equation*}
\begin{aligned}
\int_{B}\left|\nabla P V^{\varepsilon,y}\right|^2\mathrm{~d}x = &\int_{\Rn}V^{2^*} \mathrm{~d}x \, + \, O\left(\left(\frac{\varepsilon}{d}\right)^{N-2}\right) = ND+O\left(\left(\frac{\varepsilon}{d}\right)^{N-2}\right).
\end{aligned}
\end{equation*}   
Now let $\tilde \la :=\la-\frac{N(N-2)}{4}$. By employing Estimate \ref{ConfTerms}, we obtain
\begin{align*}
    \int_{B} c(x) (P  V^{\varepsilon, y})^2 \mathrm{~d}x =\frac{4\tilde \la C(N)\varepsilon^2 }{\left(1-\left|y\right|^{2}\right)^{2}} \int_{\mathbb{R}^{N}} \frac{1}{\left(1+|z|^{2}\right)^{N-2}}\mathrm{~d}z +O(\varepsilon^{N-2})+\begin{cases} O(\varepsilon^3) \text{ if } N >5\\
O(\varepsilon^3 |\log \varepsilon|) \text{ if } N =5
\end{cases}.
     \end{align*}
     Finally, using Estimate \ref{proGrad}, we have
     \begin{align*}
         \int_{B} a(x) |P  V^{\varepsilon, y}|^{2^*} \mathrm{~d}x 
         &= \int_{B} a(y) |P  V^{\varepsilon, y}|^{2^*}\mathrm{~d}x + \int_{B} (a(x)-a(y) |P  V^{\varepsilon, y}|^{2^*}\mathrm{~d}x \\
         & = a(y)\int_{B} (V^{\varepsilon, y})^{2^*} \mathrm{~d}x + O\left(\int_{B} (V^{\varepsilon, y})^{2^*-1}\varphi^{\varepsilon,y} \mathrm{~d}x \right)+\int_{B} (a(x)-a(y) |V^{\varepsilon, y}|^{2^*}\mathrm{~d}x\\
         &= a(y)\int_{\Rn} |V|^{2^*} \mathrm{~d}x+O\left(\left(\frac{\varepsilon}{d}\right)^{N}\right)+O\left(\left(\frac{\varepsilon}{d}\right)^{N-2}\right)+ \int_{B} (a(x)-a(y)) |V^{\varepsilon, y}|^{2^*}\mathrm{~d}x\\
         &= a(y)N D + O\left(\left(\frac{\varepsilon}{d}\right)^{N-2}\right)+\int_{B} (a(x)-a(y)) |V^{\varepsilon, y}|^{2^*} \,\mathrm{~d} x.
         \end{align*}
Now, we can estimate the integral in the above expression similar to in Theorem \ref{Exi-res1}. However, in this case, the point $y$ may not necessarily be the maximum point of the function $a$. To perform the required estimation, consider the following function where $\psi \in C_c^{\infty}(B)$ serves as a cut-off function around $y$ and is even with respect to $y.$ Readily, we have 
  \begin{align*}
             \int_{B}\langle \nabla a(y),z-y\rangle (\psi(z) V^{\varepsilon,y}(z))^{2^{*}} \mathrm{~d}z = 0,
  \end{align*}
and using estimates similar to those in Theorem \ref{Exi-res1} will lead us to
\begin{align*}
    \int_{B} (a(x)-a(y)) |V^{\varepsilon, y}|^{2^*} \mathrm{~d}x &= O(\varepsilon^2).
\end{align*}
Therefore, we can conclude the lemma similarly to Theorem \ref{Exi-res1} and using the continuity of $a$.
\end{proof}
\begin{lemma} \label{delta_dep}
For $N > 6$ and $\lambda < \frac{N(N-2)}{4} \, + \, \frac{1}{N + 4},$  there exist constants $C^{+}>0, 0<\varepsilon_2^{+}<\varepsilon_1^{+}$, and $0<\delta_2^{+}<\nu$ such that if $\delta \geq C^{+} \varepsilon^{1+\kappa}$, $y-\bar{y}=\varepsilon \zeta$ for some $\zeta \in B(0,1)$, then
\begin{equation*}
\partial_\delta \tilde{\mathcal J}_a\left(P  V^{\varepsilon, y}+\delta v\right)>0
\end{equation*}
holds for all $\varepsilon<\varepsilon_2^{+}, \delta<\delta_2^{+}$, $0<\kappa<1$ and $v$ satisfying \eqref{OrthCondn} with $\|v\|_{H_{0}^{1}(B(0,1))}=1$.
\end{lemma}
\begin{proof} Recall that
\begin{align*}
\tilde{\mathcal J}_a(P  V^{\varepsilon, y}+\delta v)=&\frac{1}{N}\left(\underbrace{\int_{B}\left[|\nabla (P  V^{\varepsilon, y}+\delta v)|^2- c(x) (P  V^{\varepsilon, y}+\delta v)^2\right] \mathrm{~d}x}_{A}\right)^{\frac{N}{2}}\times\\
&\qquad\left(\underbrace{\int_{B} a(x) |P  V^{\varepsilon, y}+\delta v|^{2^*} \mathrm{~d}x}_{\tilde A}\right)^{\frac{(2-N)}{2}}.
\end{align*}   
Therefore, for $u:=P  V^{\varepsilon, y}+\delta v$,
\begin{align*}
\partial_{\delta} \tilde{\mathcal J}_a(u)=& \frac{1}{2} A^ {(N-2) / 2} \tilde A^{(2-N)/2}\partial_\delta A - \frac{1}{2^*}A^{N/2} \tilde A^{-N/2} \partial_\delta \tilde A.
\end{align*}
Now we will estimate $A,\tilde A,\partial_\delta A, \partial_\delta \tilde A.$ Let's start by estimating $ A.$ 
\begin{align*}
    A&:= \int_{B}\left[|\nabla (P  V^{\varepsilon, y}+\delta v)|^2- c(x) (P  V^{\varepsilon, y}+\delta v)^2\right]\mathrm{~d}x\\
    &= \int_{B}\left(|\nabla (P  V^{\varepsilon, y})|^2+ |\nabla (\delta v)|^2 +2 \nabla (P  V^{\varepsilon, y}). \nabla (\delta v)- c(x) \left[(P  V^{\varepsilon, y})^2+ (\delta v)^2+ 2P  V^{\varepsilon, y} \delta v \right]\right) \mathrm{~d}x \\
    &=ND+O\left(\left(\frac{\varepsilon}{d}\right)^{N-2}\right)+ O(\delta^2)+O(\varepsilon^2)+\delta\begin{cases}
O\left(\varepsilon^{\frac{3}{2}}\right) & \text { if } N=5,\\
O\left(\varepsilon^2 \left(\log \left(\frac{1}{\varepsilon}\right)\right)^{2 / 3}\right) & \text { if } N=6,\\
O\left(\varepsilon^{2}\right) & \text { if } N>6. 
\end{cases}
\end{align*}
In the final step, we applied Estimate \ref{proGrad}, \eqref{ConfTerm1} from Estimate \ref{ConfTerms}, the orthogonality condition for $v$ with $\|v\|_{H_{0}^{1}(B(0,1))}=1$, and the Hardy inequality in hyperbolic space \cite[Corollary 2.2]{BGG}. Moreover, to obtain the above estimate of the last integral, note that
\begin{align*}
    \left|\int_{B}c(x) (P  V^{\varepsilon, y})v\mathrm{~d}x\right|\leq \left|\int_{B(y,r)}\right|+\left|\int_{B\setminus B(y,r)}\right|,
\end{align*}
where $r$ is small. We can eliminate the singularity for the integral near the point $y$ and apply Estimate \eqref{vTerm1}. For the integral outside the ball, we use H\"older's inequality followed by Hardy's inequality, as demonstrated in the proof of Estimate \ref{ConfTerm1}.\\ \\
Now consider 
\begin{align*}
    \tilde A & :=\int_{B} a(x) |P  V^{\varepsilon, y}+\delta v|^{2^*} \mathrm{~d}x
    = \int_{B} a(x)(P  V^{\varepsilon, y})^{2^*} \mathrm{~d}x   + 2^*\int_{B} a(x)(P  V^{\varepsilon, y})^{2^*-1}\delta v \mathrm{~d}x \\
& \hspace{5.2cm}+ O\left(\int_{B}a(x)(P  V^{\varepsilon, y})^{2^*-2}\delta^2 v^{2}+ a(x)|\delta v|^{2^*} \mathrm{~d}x \right).
\end{align*}
Considering each of the terms individually
\begin{align*}
    \int_{B} a(x)|P  V^{\varepsilon, y}|^{2^*} \mathrm{~d}x   &= \int_{B} a(y)|P  V^{\varepsilon, y}|^{2^*} \mathrm{~d}x+\int_{B} (a(x)-a(y))|P  V^{\varepsilon, y}|^{2^*} \mathrm{~d}x\\
    &= a(y)ND + O\left(\left(\frac{\varepsilon}{d}\right)^{N-2}\right)+O(\varepsilon^{2}) \text{ (following Lemma \ref{energyMain})}.
\end{align*}
Now, applying Estimate \ref{vTerm2}, we obtain
\begin{align*}
2^*\int_{B} a(x) |P  V^{\varepsilon, y}|^{2^*-1}|
\delta v| \mathrm{~d}x =O(\varepsilon^{\frac{\theta(N+2)}{2N}} \delta).
\end{align*}
Finally, applying H\"older inequality to the remaining terms yields
\begin{align*}
\int_{B} a(x) |P  V^{\varepsilon, y}|^{2^*-2}|
\delta v|^2 \mathrm{~d}x & \leq \|a\|_{\infty} \delta^2 \left[\int_{\Rn}\left( |V^{\varepsilon, y}|^{2^*-2}\right)^{\frac{2^*}{2^*-2}}\mathrm{~d}x\right]^{\frac{2^*-2}{2^*}} \left[\int_{B}|v|^{2^*} \mathrm{~d}x\right]^\frac{2}{2^*} =O(\delta^2).
\end{align*}
Consequently, we derive 
\begin{align*}
    \tilde A= a(y)ND +O\left(\left(\frac{\varepsilon}{d}\right)^{N-2}\right)+ O(\varepsilon ^2)+ O(\varepsilon^{\frac{\theta(N+2)}{2N}}\delta)+O(\delta ^2 )+ O(\delta^{2^*}).
\end{align*}
Using these estimates, we find
\begin{align*}
\partial_\delta \tilde{\mathcal J}_a(u) &= \frac{1}{2}\left(N D+o(1)\right)^{(N-2) / 2}\left(a(y) N D+o(1)\right)^{(2-N) / 2} \partial_\delta A \\
&\qquad-\frac{1}{2^*}\left(N D+o(1)\right)^{N / 2}\left(a(y) N D+o(1)\right)^{-N / 2} \partial_\delta \tilde A,
\end{align*}
as $(\varepsilon, \delta) \rightarrow 0$.\\
Now estimating $\partial_\delta A$ and $\partial_\delta \tilde A$: Using the expression of $A$, it follows that
\begin{align*}
&\partial_\delta A=2 \delta\left(1-\int_{B} c(x)v^2 \mathrm{~d}x\right)- 2\int_{B}c(x) (P  V^{\varepsilon, y})v\mathrm{~d}x\\
&= 2 \delta\left(1-\int_{B} c(x)v^2\mathrm{~d}x\right)+ \begin{cases}
O\left(\varepsilon^{\frac{3}{2}}\right) & \text { if } N=5,\\
O\left(\varepsilon^2 \left(\log \left(\frac{1}{\varepsilon}\right)\right)^{2 / 3}\right) & \text { if } N=6,\\
O\left(\varepsilon^{2}\right) & \text { if } N>6.
\end{cases} 
\end{align*}
By applying Estimate \ref{vTerm2}, we can easily deduce
\begin{align*}
\partial_\delta \tilde A&= 2^*(2^* -1)\delta\int_{B}a(x) |P  V^{\varepsilon, y}|^{2^*-2}v^2 \mathrm{~d}x + O(\varepsilon^{\frac{\theta(N+2)}{2N}})+ o(\delta).
\end{align*}
Estimating the integral in the above expression yields
\begin{align*}
    & \int_{B} a(x)\delta|P  V^{\varepsilon, y}|^{2^*-2}v^2\mathrm{~d}x 
    =  \int_{B} a(y)\delta|P  V^{\varepsilon, y}|^{2^*-2}v^2 \mathrm{~d}x + \int_{B} (a(x)-a(y))\delta|P  V^{\varepsilon, y}|^{2^*-2}v^2\mathrm{~d}x.
    \end{align*}
    Now
    \begin{align*}
    \int_{B} |a(x)-a(y)|\delta|P  V^{\varepsilon, y}|^{2^*-2}v^2\mathrm{~d}x&\leq  c \delta \int_{B}\frac{|x-y|}{1+|x-y|} |V^{\varepsilon, y}|^{2^*-2}v^2\mathrm{~d}x\\
    &\leq c \delta \left[\int_{B}\left(\frac{|x-y|}{1+|x-y|} |V^{\varepsilon, y}|^{2^*-2}\right)^{\frac{2^*}{2^*-2}}\mathrm{~d}x\right]^{\frac{2^*-2}{2^*}} \left[\int_{B}|v|^\frac{2.2^*}{2} \mathrm{~d}x\right]^\frac{2}{2^*}\\
    &\leq c \delta \varepsilon \underbrace{\left[\int_{\Rn}|z|^{N/2} |V(z)|^{2^*}\: \mathrm{~d}z\right]^{2/N}}_{< \infty} \underbrace{\int_{B}|\nabla v(x)|^{2}\mathrm{~d}x}_{=1}\\
    &=O(\delta \varepsilon),
    \end{align*}
    where for the penultimate step, Poincar\'e's inequality has been utilized.\\
\noindent
Therefore, we get
\begin{align*}
\partial_\delta \tilde A&= 2^*(2^* -1)\delta a(y)\int_{B}|P  V^{\varepsilon, y}|^{2^*-2}v^2\mathrm{~d}x + O(\delta \varepsilon)+O(\varepsilon^{\frac{\theta(N+2)}{2N}})+ o(\delta).
\end{align*}
Also, note that $\frac{\theta(N+2)}{2N}>2$ since $\theta >4$. Thus, combining all the above estimates implies
\begin{align*}
\partial_\delta \tilde{J}(u)&=a(y)^{(2-N) / 2} \delta\left(1-\int_{B} c(x) v^2\mathrm{~d}x-(2^*-1) \int_{B} |P  V^{\varepsilon, y}|^{2^*-2}v^{2}\mathrm{~d}x\right)\\
&\hspace{4.5cm}+O(\varepsilon^2)+O(\delta \varepsilon)+O(\varepsilon^{\frac{\theta(N+2)}{2N}})+ o(\delta).
\end{align*}
Now it remains to show that the bracketed term is positive. For that, we use the inequality \cite[D.1]{ReyO} to obtain
\begin{equation*}
    1- (2^*-1) \int_{B} |P  V^{\varepsilon, y}|^{2^*-2}v^{2}\mathrm{~d}x \geq 1- (2^*-1) \int_{B} |  V^{\varepsilon, y}|^{2^*-2}v^{2}\mathrm{~d}x \geq \frac{4}{N+4} \int_{B} |\grad v|^2 \mathrm{~d}x.
\end{equation*}
Furthermore, \cite[Corollary~2.2]{BGG} implies
\begin{equation*}
(4 \lambda- N(N-2))\int_{B}|\nabla v|^2 \mathrm{~d}x\geq  \left(\lambda-\frac{N(N-2)}{4}\right)\int_{B}\left(\frac{2}{1-|x|^2}\right)^2 v^2 \mathrm{~d}x.
\end{equation*}
Therefore, for $\lambda < \frac{N(N-2)}{4}+\frac{1}{N+4}$, we get
\begin{align*}
    &1-\int_{B} c(x) v^2\mathrm{~d}x-(2^*-1) \int_{B} |P  V^{\varepsilon, y}|^{2^*-2}v^{2}\mathrm{~d}x \\
    &\geq \frac{4}{N+4} \int_{B} |\grad v|^2- (4 \lambda- N(N-2))\int_{B}|\nabla v|^2 \mathrm{~d}x >0.
    \end{align*}
    This concludes the lemma.
  \end{proof}

  \medskip
  \noindent
Let
\begin{equation*}
\mathcal{V}:=\left\{u \in H^1_{0}(B(0,1)) \text { s.t. }\left\|u-c PV^{\varepsilon, y}\right\|<\nu \text { for some } 0<\varepsilon<\nu, c>\mu \text { and } y \in B(0,1)\right\}.
\end{equation*}
In $\mathcal{V}$, define the vector field $\Upsilon$ by
\begin{equation*}
u:=c(u) P V^{\varepsilon(u), y(u)}+\delta(u) v(u) \longmapsto c(u)  \partial_{\varepsilon} PV^{\varepsilon, y}{ }_{\mid \varepsilon=\varepsilon(u), y=y(u)},
\end{equation*}
where $v$ satisfies condition \eqref{OrthCondn} and $\|v\|_{H_{0}^{1}(B(0,1))}=1$. By the definition of $\mathcal{V}$, the vector field $\Upsilon$ is well-defined and exhibits local Lipschitz continuity, even if unbounded as $\varepsilon(u) \rightarrow 0$. Let $\eta(., u)$ represent the flow generated by $\Upsilon$ starting from the initial condition $u$. Initially, $\eta$ is defined only locally.

\medskip

\begin{lemma} \label{defo}
     Let $u_0 \in \mathcal{V}$. Then $\eta\left(., u_0\right)$ is defined as long as $\varepsilon\left(\eta\left(., u_0\right)\right)<\nu$, and
     \begin{enumerate}
     \item  $c\left(\eta\left(., u_0\right)\right)=c\left(u_0\right)$,
 \item $\delta\left(\eta\left(., u_0\right)\right)=\delta\left(u_0\right)$,
 \item  $\left\|y\left(\eta\left(t, u_0\right)\right)-y\left(u_0\right)\right\| \leq t O\left(\delta\left(u_0\right)\right)$,
 \item  $\partial_t \varepsilon\left(\eta\left(t, u_0\right)\right)=1+O\left(\delta\left(u_0\right)\right)$.
\end{enumerate}
\end{lemma}
\begin{proof}
The proof follows along the lines of \cite[Lemma~5.7]{Sm}. We shall briefly sketch the proof. Consider the curve $\tilde{\eta}$ defined by
\begin{equation}
\tilde{\eta}(t):=c\left(u_0\right) PV^{\varepsilon\left(u_0\right)+t, y\left(u_0\right)}+\delta\left(u_0\right) v\left(u_0\right), \label{PertCur}
\end{equation}
with $\tilde{c}, \tilde{\varepsilon}, \tilde{y}, \tilde{\delta},$ and $\tilde{v}$ being its coordinate functions.\\
\noindent
Since $\partial_t \tilde{\eta}(0)=\Upsilon\left(u_0\right)$, the derivative $\partial_t \varepsilon\left(\eta\left(t, u_0\right)\right)_{\mid t=0}$  coincides with $\partial_t \tilde{\varepsilon}(0)$, and this holds for the other coordinates as well. Indeed, for $\eta\left(., u_0\right)$ to be defined as a flow from $\mathbb{R} \times \mathcal{V} \rightarrow \mathcal{V}$, we have 
\begin{equation*}
 \eta\left(t, u_0\right)= c(\eta\left(t, u_0\right))PV^{\varepsilon(\eta\left(t, u_0\right)), y(\eta\left(t, u_0\right))}+\delta(\eta\left(t, u_0\right)) v(\eta\left(t, u_0\right)).  
\end{equation*}
Differentiating the above equation with respect to $t$, we have
\begin{align} \label{eq1}
 \partial_t \eta\left(0, u_0\right)&= \partial_t c(\eta\left(t, u_0\right))_{\mid t=0} PV^{\varepsilon(u_0), y(u_0)}+ c(u_0) \partial_{\varepsilon}PV^{\varepsilon(u_0), y(u_0)} \partial_t \varepsilon(\eta\left(t,u_0\right))_{\mid t=0}\\ \notag
 & \hspace{-0.7cm}+ c(u_0) \left\langle\nabla_y PV^{\varepsilon( u_0), y(u_0)}, \partial_{t} y(\eta\left(t, u_0\right))_{\mid t=0}\right\rangle +\partial_{t} \delta(\eta\left(t, u_0\right))_{\mid t=0} v(u_0)
    + \delta(u_0) \partial_t v(\eta\left(t, u_0\right))_{\mid t=0}. \notag
\end{align}
Also, the flow $\eta$ satisfies the following ODE:
\begin{equation}
    \partial_t \eta\left(t, u_0\right)= \Upsilon\left(\eta\left(t, u_0\right)\right), \;\;\;\;
    \eta\left(0, u_0\right)= u_0. \label{eqq3}
\end{equation}
\noindent
Now differentiating $\tilde{\eta}$ with respect to $t$ gives
\begin{align}  \label{eq2}
    \partial_t \tilde{\eta}(0)
      &=\partial_t\tilde{c}(0) PV^{\varepsilon\left(u_0\right), y\left(u_0\right)}+ c\left(u_0\right)  \partial_{\varepsilon}PV^{\varepsilon\left(u_0\right), y\left(u_0\right)}\partial_t \tilde{\varepsilon}(0) \notag \\
     & + c\left(u_0\right)\left\langle\nabla_{y} PV^{\varepsilon\left(u_0\right), y\left(u_0\right)}, \partial_t \tilde{y}(0)\right\rangle + \partial_{t} \tilde{\delta}(0) v(u_0)+  \delta(u_0) \partial_t \tilde{v}(0).
\end{align}
Thus, keeping in mind \eqref{eqq3} and comparing \eqref{eq1} and \eqref{eq2}, we find that the derivative $\partial_t \varepsilon\left(\eta\left(t, u_0\right)\right)_{\mid t=0}$ is the same as $\partial_t \tilde{\varepsilon}(0)$, and this equivalence holds for the other coordinates as well.

\medskip
\noindent
The rest of the proof hinges on the orthogonality condition of $v,$ and on estimates \ref{proGrad}, \eqref{ProDely}, and \eqref{ProjDelEp} from Estimate \ref{Proj&ProjDer}. For brevity, we skip the details.
\end{proof}

\medskip
\begin{lemma} \label{epsiChan}
   Let $C^{+}$ and $y$ be as specified by Lemma \ref{delta_dep}, and consider $N >6.$ There exist $0<\varepsilon_3^{+}<$ $\varepsilon_2^{+}$, $0<\kappa^{\prime}<1$ and $0<\delta_3^{+}<\delta_2^{+}$ such that
\begin{equation*}
\left(\delta \leq 2 C^{+} \varepsilon^{1+\kappa^{\prime}}\right) \Longrightarrow \partial_t \tilde{\mathcal J}_a\left(\eta\left(t, PV^{\varepsilon, y}+\delta v\right)\right)_{\mid t=0}<0
\end{equation*}
for all $\varepsilon<\varepsilon_3^{+}, \delta<\delta_3^{+}$and $v \in {H}^{1}_{0}\left(B(0,1)\right)$ that satisfies \eqref{OrthCondn} with $\|v\|_{H_{0}^{1}(B(0,1))}=1.$
\end{lemma}
\begin{proof}
To begin, observe that
\begin{equation*}
\partial_t \tilde{\mathcal J}_a\left(\eta\left(t, PV^{\varepsilon, y}+\delta v\right)\right)_{\mid t=0}=\partial_{\varepsilon}  \tilde{\mathcal J}_a\left(PV^{\varepsilon, y}+\delta v\right)_{\mid \varepsilon=\varepsilon}.
\end{equation*}
For this proof, we will refer to the ball $B(y,d)$ simply as $B_d$.

\medskip

\noindent
As detailed in Lemma \ref{delta_dep}, we have
\begin{equation*}
\begin{aligned}
\partial_{\varepsilon} \tilde{\mathcal J}_a\left(u\right)= & \frac{1}{2}\left(N D+o(1)\right)^{(N-2) / 2}\left(a(y) N D+o(1)\right)^{(2-N) / 2} \partial_{\varepsilon} A \\
&\qquad-\frac{1}{2^*}\left(N D+o(1)\right)^{N / 2}\left(a(y) N D+o(1)\right)^{-N / 2} \partial_{\varepsilon} \tilde{A},
\end{aligned}
\end{equation*}
when $(\varepsilon, \delta) \rightarrow 0$, where $A$ and $\tilde{A}$ were defined in Lemma \ref{delta_dep} and $u:=PV^{\varepsilon, y}+\delta v$.\\
Remember that
\begin{align*}
    A&:= \int_{B}\left(|\nabla (P  V^{\varepsilon, y}+\delta v)|^2- c(x) (P  V^{\varepsilon, y}+\delta v)^2\right)\mathrm{~d}x.
\end{align*}
Expanding this, we get
\begin{equation*}
\begin{aligned}
\partial_{\varepsilon} A & =\underbrace{\partial_{\varepsilon}\int_{B}|\nabla (P  V^{\varepsilon, y})|^2 \mathrm{~d}x}_{A1}- \underbrace{\partial_{\varepsilon}\int_{B}c(x) (P  V^{\varepsilon, y})^2\mathrm{~d}x}_{A2}- 2\underbrace{\partial_{\varepsilon}\int_{B}c(x)P  V^{\varepsilon, y} \delta v \mathrm{~d}x}_{A3}.
\end{aligned}
\end{equation*}
Next, we will estimate each term in turn. 

\begin{align*}
      \int_{B}\left|\nabla P V^{\varepsilon,y}\right|^2 \mathrm{~d}x \; = \; \int_{B}(V^{\varepsilon,y})^{2^{*}-1}P V^{\varepsilon,y}\mathrm{~d}x =  \int_{B}(V^{\varepsilon,y})^{2^{*}}\mathrm{~d}x-\int_{B}(V^{\varepsilon,y})^{2^{*}-1}\varphi^{\varepsilon,y}\mathrm{~d}x.
    \end{align*}
Thus
 \begin{align*}
    A1 &:= \partial_{\varepsilon}\int_{B}|\nabla (P  V^{\varepsilon, y})|^2 \mathrm{~d}x=  \partial_{\varepsilon}\left[\int_{B}(V^{\varepsilon,y})^{2^{*}}\mathrm{~d}x-\int_{B}(V^{\varepsilon,y})^{2^{*}-1}\varphi^{\varepsilon,y}\mathrm{~d}x\right]\\
    &=  \partial_{\varepsilon}\left[\int_{\Rn}(V^{\varepsilon,y})^{2^{*}}\mathrm{~d}x-\int_{\Rn\setminus B}(V^{\varepsilon,y})^{2^{*}}\: \mathrm{~d}x-\int_{B}(V^{\varepsilon,y})^{2^{*}-1}\varphi^{\varepsilon,y}\mathrm{~d}x\right]\\
     &=-\partial_{\varepsilon}\int_{\Rn\setminus B}(V^{\varepsilon,y})^{2^{*}}\mathrm{~d}x-\partial_{\varepsilon}\int_{B}(V^{\varepsilon,y})^{2^{*}-1}\varphi^{\varepsilon,y}\mathrm{~d}x.
    \end{align*}
Moreover,
\begin{align*}
  \partial_{\varepsilon}\int_{\Rn\setminus B}(V^{\varepsilon,y})^{2^{*}} \mathrm{~d}x&= 2^* \int_{\Rn\setminus B}(V^{\varepsilon,y})^{2^{*}-1} \frac{\partial V^{\varepsilon, y}}{\partial \varepsilon} \mathrm{~d}x \\
  &\leq C(N) \int_{\Rn\setminus B_d}
\frac{\left[|\frac{x-y}{\varepsilon}|^2-1\right]}{\varepsilon^{N+1}\left[1+|\frac{x-y}{\varepsilon}|^2\right]^{N+1}} \mathrm{~d}x\\
&=O\left(\frac{\varepsilon^{N-1}}{d^N}\right).
\end{align*}
Next,
\begin{align*}
\partial_{\varepsilon}\int_{B}(V^{\varepsilon,y})^{2^{*}-1}\varphi^{\varepsilon,y} \mathrm{~d}x=  \underbrace{\int_{B}(V^{\varepsilon,y})^{2^{*}-1}\frac{\partial \varphi^{\varepsilon, y}}{\partial \varepsilon} \mathrm{~d}x}_{A1.2a}\; + \; (2^*-1)\underbrace{\int_{B}(V^{\varepsilon,y})^{2^{*}-2}\frac{\partial V^{\varepsilon, y}}{\partial \varepsilon}\varphi^{\varepsilon,y} \mathrm{~d}x.}_{A1.2b}
    \end{align*}
First, we can evaluate $A1.2a$ using Estimate \ref{ErrorPro} in the following manner:
\begin{align*}
    &\int_{B}(V^{\varepsilon,y})^{2^{*}-1}\frac{\partial \varphi^{\varepsilon, y}}{\partial \varepsilon} \mathrm{~d}x=\int_{B} (V^{\varepsilon,y})^{2^{*}-1} \left[\left(\frac{N-2}{2}\right)\varepsilon^{\frac{N-4}{2}} H(y, x)+\frac{\partial f^{\varepsilon,y}}{\partial \varepsilon}\right]\mathrm{~d}x \\
&= \frac{(N-2)}{2}\underbrace{\int_{B_d} (V^{\varepsilon,y})^{2^{*}-1} \varepsilon^{\frac{N-4}{2}} H(y,x)}_{a1} \mathrm{~d}x \; +\;\frac{(N-2)}{2}\underbrace{\int_{B \setminus B_d} (V^{\varepsilon,y})^{2^{*}-1} \varepsilon^{\frac{N-4}{2}} H(y, x) \mathrm{~d}x}_{a2}\\
&\hspace{10cm}+ \underbrace{\int_{B} (V^{\varepsilon,y})^{2^{*}-1}\frac{\partial f^{\varepsilon,y}}{\partial \varepsilon}\mathrm{~d}x.}_{a3}
\end{align*}
Utilizing Estimate \ref{ErrorPro} and \eqref{aubinCritminusone}, we find that
\begin{align*}
 a3:= \int_{B} (V^{\varepsilon,y})^{2^{*}-1}\frac{\partial f^{\varepsilon,y}}{\partial \varepsilon} \mathrm{~d}x \; =\; O\left(\frac{\varepsilon^{\frac{N}{2}}}{ d^{N}}\varepsilon^{\frac{N}{2}-1}\int_{\Rn} V^{2^*-1} \mathrm{~d}x \right) = O \left(\frac{\varepsilon^{N-1}}{d^{N}}\right).
\end{align*}
Furthermore, considering the integral over $B \setminus B_d$, we first note, based on harmonicity, that
\begin{equation}
H(y,x) \leq \max_{\overline{B}} H\left(y, \cdot\right)=\max _{\partial B} H\left(y, \cdot\right) \lesssim \frac{1}{d^{N-2}}.
\end{equation}
Using this inequality along with \eqref{aubinCritminusone}, we obtain
\begin{align*}
    a2:= \int_{B \setminus B_d} (V^{\varepsilon,y})^{2^{*}-1} \varepsilon^{\frac{N-4}{2}} H(y, x) \mathrm{~d}x&= O\left(\frac{\varepsilon^{\frac{N-4}{2}}}{d^{N-2}}\int_{\Rn\setminus B_d}(V^{\varepsilon,y})^{2^{*}-1} \mathrm{~d}x \right)=O\left(\frac{\varepsilon^{N-1}}{d^{N}}\right).
\end{align*}
Finally, we need to estimate
\begin{align*}
   a1 := &\int_{B_d} (V^{\varepsilon, y})^{2^* - 1} \varepsilon^{\frac{N-4}{2}} H(y, x) \, \mathrm{d}x \\
    &= \varepsilon^{\frac{N-4}{2}} H(y, y) \int_{B_d} (V^{\varepsilon, y})^{2^* - 1} \, \mathrm{d}x \; + \; \varepsilon^{\frac{N-4}{2}} \int_{B_d} \left[ H(y, x) - H(y, y) \right] (V^{\varepsilon, y})^{2^* - 1} \, \mathrm{d}x.
\end{align*}

Expanding $H(y, \cdot)$ in a Taylor series up to the third order around $y$, we have $H(y, x) = H(y, y) + H_1 + H_2 + H_3 + R$,
where $H_j$ denotes the $j$-th order term (e.g., $H_1 = \nabla H(y, y) \cdot (x - y)$). It can be seen that the integrals involving $H_j$ vanish: $H_1$ and $H_3$ are zero due to symmetry, and $H_2$ is zero due to harmonicity because the Hessian matrix, being symmetric, can be diagonalized orthogonally. Additionally, the remainder term $|R|$ is bounded by $\sup_{B_d} \left\|\nabla^4 H(y, x) \right\| |x - y|^4$. Using the explicit form of $H$, it is easy to verify that

$$
\sup _{B_d}\left\|\nabla^{4} H\left(y, \cdot\right)\right\|\leq \frac{C}{d^{N+2}},
$$
leading to the estimate

\begin{equation*}
 \varepsilon^{\frac{N-4}{2}} \left|\int_{B_d}(V^{\varepsilon,y})^{2^{*}-1}R \mathrm{~d}x\right| \leq \frac{C \varepsilon^{\frac{N-4}{2}}}{d^{N+2}} \int_{B_d}(V^{\varepsilon,y})^{2^{*}-1}\left|x-y\right|^{4}\;\mathrm{~d} x .
\end{equation*}
This is then evaluated using scaling arguments and spherical coordinates
\begin{equation*}
\begin{aligned}
\int_{B_d}(V^{\varepsilon,y})^{2^{*}-1}\left|x-y\right|^{4}\mathrm{~d}x & = C\varepsilon^{\frac{N}{2}+3}\int_{0}^{\frac{d}{\varepsilon}} \frac{r^{N+3}}{\left(1+r^{2}\right)^{\frac{N+2}{2}}}\;\mathrm{~d}r \leq C \varepsilon^{\frac{N}{2}+3} \left(\frac{d}{\varepsilon}\right)^2 \approx d^2 \varepsilon^{\frac{N}{2}+1} .
\end{aligned}
\end{equation*}
Therefore, using \eqref{aubinCritminusone}, we obtain
\begin{align*}
    a1&= \varepsilon^{\frac{N-4}{2}} H(y,y)\int_{B_d} (V^{\varepsilon,y})^{2^{*}-1} \mathrm{~d}x \; + \; \varepsilon^{\frac{N-4}{2}}\int_{B_d}\left[ H(y, x)- H(y, y)\right] (V^{\varepsilon,y})^{2^{*}-1}\mathrm{~d}x \\
    &=\varepsilon^{N-3}H(y,y)\int_{\Rn} V^{2^*-1} \mathrm{~d}x \;  + \; O\left(\frac{\varepsilon^{N-1}}{d^2}\right)+O\left(\frac{\varepsilon^{N-1}}{d^N}\right).
    \end{align*}
   Thus
   \begin{align*}
       A1.2a:= \frac{(N-2)}{2}\varepsilon^{N-3}H(y,y)\int_{\Rn} V^{2^*-1}\mathrm{~d}x \;  + \; O\left(\frac{\varepsilon^{N-1}}{d^2}\right)+O\left(\frac{\varepsilon^{N-1}}{d^N}\right).
   \end{align*}
   Next, applying Estimate \ref{ErrorPro}, $A1.2b$ becomes
   \begin{align*}
     A1.2b &:=\int_{B}(V^{\varepsilon,y})^{2^{*}-2}\frac{\partial 
     V^{\varepsilon, y}}{\partial \varepsilon}\varphi^{\varepsilon,y}\mathrm{~d}x \\
      &=\underbrace{\int_{B_d}(V^{\varepsilon,y})^{2^{*}-2}\frac{\partial V^{\varepsilon,y}}{\partial \varepsilon}\varepsilon^{\frac{N-2}{2}}H(y, x) \mathrm{~d}x}_{b1}+\underbrace{\int_{B\setminus B_d}(V^{\varepsilon,y})^{2^{*}-2}\frac{\partial V^{\varepsilon,y}}{\partial \varepsilon}\varepsilon^{\frac{N-2}{2}}H(y, x) \mathrm{~d}x}_{b2}\\
      &\qquad+\underbrace{\int_{B}(V^{\varepsilon,y})^{2^{*}-2}\frac{\partial V^{\varepsilon,y}}{\partial \varepsilon}f^{\varepsilon,y} \mathrm{~d}x.}_{b3}
      \end{align*}
First, we estimate
      \begin{align*}
      \int_{B}(V^{\varepsilon,y})^{2^{*}-2}\frac{\partial V^{\varepsilon,y}}{\partial \varepsilon} \mathrm{~d}x &= C(N)\int_{B}\frac{1}{\varepsilon^{\frac{N}{2}+2}}\frac{\left(\frac{|x-y|^2}{\varepsilon^2}-1\right)}{\left[1+\frac{|x-y|^2}{\varepsilon^2}\right]^{\frac{N}{2}+2}} \mathrm{~d}x\\
      &=C(N)\varepsilon^{\frac{N}{2}-2}\int_{\Rn}\frac{(|z|^2-1)}{(1+|z|^2)^{\frac{N}{2}+2}} \mathrm{~d}z+O\left(\varepsilon^{\frac{N}{2}-2}\int_{\frac{d}{\varepsilon}}^{\infty}\frac{(r^2-1)r^{N-1}}{(1+r^2)^{\frac{N}{2}+2}} \;\mathrm{~d}r\right)\\
      &=C(N)\varepsilon^{\frac{N}{2}-2}\int_{\Rn}\frac{(|z|^2-1)}{(1+|z|^2)^{\frac{N}{2}+2}} \mathrm{~d}z+ O\left(\frac{\varepsilon^{\frac{N}{2}}}{d^2}\right).
      \end{align*}
Consequently, using Estimate \ref{ErrorPro}, we have
\begin{align*}
 b3:= \int_{B}(V^{\varepsilon,y})^{2^{*}-2}\frac{\partial V^{\varepsilon,y}}{\partial \varepsilon}f^{\varepsilon,y} \mathrm{~d}x \; =\; O\left(\frac{\varepsilon^{\frac{N+2}{2}}}{ d^{N}}\int_{\Rn}(V^{\varepsilon,y})^{2^{*}-2}\frac{\partial V^{\varepsilon,y}}{\partial \varepsilon} \mathrm{~d}x \right) =O\left(\frac{\varepsilon^{N-1}}{d^N}\right).
\end{align*}
Employing methods analogous to those utilized for analyzing $a2$ and $a1$, we can similarly evaluate $b2$ and $b1$, leading to the following estimate
\noindent
\begin{align*}
    A1.2b=C(N)\varepsilon^{N-3}H(y,y)\int_{\Rn}\frac{(|z|^2-1)}{(1+|z|^2)^{\frac{N}{2}+2}} \mathrm{~d}z+O\left(\frac{\varepsilon^{N-1}}{d^2}\right)+O\left(\frac{\varepsilon^{N-1}}{d^N}\right).
\end{align*}
Combining all the derived estimates, we obtain
\begin{align*}
    A1&= - \frac{(N-2)}{2}\varepsilon^{N-3}H(y,y)\int_{\Rn} V^{2^*-1} \; {\rm d}z  -C(N)\varepsilon^{N-3}H(y,y)\int_{\Rn}\frac{(|z|^2-1)}{(1+|z|^2)^{\frac{N}{2}+2}} \; {\rm d}z \\&+O\left(\frac{\varepsilon^{N-1}}{d^2}\right)
   +O\left(\frac{\varepsilon^{N-1}}{d^N}\right).
\end{align*}

\medskip
\noindent
Turning to $A2$, applying \eqref{ConfTerm3} from Estimate \ref{ConfTerms}, we get
\begin{align*}
   A2:=\partial_{\varepsilon}\int_{B}c(x) (P  V^{\varepsilon, y})^2\mathrm{~d}x &= 2\int_{B}c(x) P  V^{\varepsilon, y}\frac{\partial PV^{\varepsilon,y}}{\partial \varepsilon} \mathrm{~d}x \\
    &=\frac{C(N)\tilde{\lambda}\varepsilon}{(1-|y|^2)^2}\int_{\Rn}\frac{(|z|^2-1)}{(1+|z|^2)^{N-1}} \mathrm{~d}z+ O\left(\varepsilon^{N-3}\right)+ \begin{cases} O(\varepsilon^2) \text{ if } N >5\\
O(\varepsilon^{2} |\log \varepsilon|) \text{ if } N =5.\end{cases}
\end{align*}

\medskip
\noindent
Finally, by estimating $A3$ as in Lemma \ref{delta_dep} by replacing $P V^{\varepsilon, y}$ by $\partial_{\varepsilon}P  V^{\varepsilon, y}$, we get
\begin{align*}
    A3&:= \partial_{\varepsilon}\int_{B}c(x)P  V^{\varepsilon, y} \delta v \mathrm{~d}x= \delta\begin{cases}
O\left(\varepsilon^{1/2}\right) & \text { if } N=5,\\
O\left(\varepsilon |\log \varepsilon|^{2/3}\right) & \text { if } N=6,\\
O\left(\varepsilon\right) & \text { if } N>6.
\end{cases}
\end{align*}
Combining all the necessary estimates, we obtain
\begin{align*}
 \partial_{\varepsilon} A &=-\frac{C(N)\tilde{\lambda}\varepsilon}{(1-|y|^2)^2}\int_{\Rn}\frac{(|z|^2-1)}{(1+|z|^2)^{N-1}} \mathrm{~d}z+ \begin{cases} O(\varepsilon^2) \text{ if } N >5\\
O(\varepsilon^{2} |\log \varepsilon|) \text{ if } N =5\end{cases}\hspace{-0.3cm}+ \delta\begin{cases}
O\left(\varepsilon^{1/2}\right) & \text { if } N=5,\\
O\left(\varepsilon |\log \varepsilon|^{2/3}\right) & \text { if } N=6,\\
O\left(\varepsilon\right) & \text { if } N>6.
\end{cases}
\end{align*}

\medskip
\noindent
Further, recall
\begin{align*}
    \tilde{A} & :=\int_{B} a(x) |P  V^{\varepsilon, y}+\delta v|^{2^*}  \mathrm{~d}x.
\end{align*}
Then 
\begin{align*}
    \partial_{\varepsilon} \tilde{A}&= 2^*\int_{B} a(x) |P  V^{\varepsilon, y}+\delta v|^{2^*-2}(P  V^{\varepsilon, y}+\delta v)\partial_{\varepsilon} P  V^{\varepsilon, y} \mathrm{~d}x\\
    & = 2^*\int_{B} a(x) \partial_{\varepsilon} PV^{\varepsilon, y}\left[(P  V^{\varepsilon, y})^{2^*-1}+ (2^*-1)(P  V^{\varepsilon, y})^{2^*-2}
\delta v\right]\mathrm{~d}x \\
&\hspace{1.3cm}+O\left(\int_{B}\left|\partial_{\varepsilon} P  V^{\varepsilon, y}\right|\left(P  V^{\varepsilon, y}\right)^{2^*-3}|\delta v|^{2}+\left|\partial_{\varepsilon} P  V^{\varepsilon, y}\right||\delta v|^{2^*-1}\mathrm{~d}x\right).
\end{align*}
Let's analyze each term individually.

\medskip
\noindent
For the following integral, we use the expansion $$(P  V^{\varepsilon, y})^{2^*-1}=(V^{\varepsilon, y})^{2^*-1}-(2^*-1)(V^{\varepsilon, y})^{2^*-2}\phi^{\varepsilon, y}+O\left((V^{\varepsilon, y})^{2^*-3}(\phi^{\varepsilon, y})^2\right), $$
and perform the calculations similar to those done in Estimate \ref{vTerm2} to obtain
\begin{align*}
   \int_{B} a(x)\partial_{\varepsilon} PV^{\varepsilon, y}(P  V^{\varepsilon, y})^{2^*-1}\mathrm{~d}x &= \int_{B} a(x)\partial_{\varepsilon} V^{\varepsilon, y}(P  V^{\varepsilon, y})^{2^*-1}\mathrm{~d}x\\
   &\qquad-\int_{B} a(x)\partial_{\varepsilon} \phi^{\varepsilon, y}(P  V^{\varepsilon, y})^{2^*-1}\mathrm{~d}x\\
   &=\int_{B} a(x)\partial_{\varepsilon} V^{\varepsilon, y}(V^{\varepsilon, y})^{2^*-1}\mathrm{~d}x \\
   &\qquad-(2^*-1)\int_{B} a(x)\partial_{\varepsilon} V^{\varepsilon, y}(V^{\varepsilon, y})^{2^*-2}\phi^{\varepsilon, y}\mathrm{~d}x\\
   &\qquad-\int_{B} a(x)\partial_{\varepsilon} \phi^{\varepsilon, y}(P  V^{\varepsilon, y})^{2^*-1}\mathrm{~d}x\\
    &\qquad+O\left(\int_{B} \left|\partial_{\varepsilon} V^{\varepsilon, y}\right|(V^{\varepsilon, y})^{2^*-3}(\phi^{\varepsilon, y})^2\right)\\
    &= -a(\bar{y})\int_{\Rn\setminus B}\partial_{\varepsilon} V^{\varepsilon, y}(V^{\varepsilon, y})^{2^*-1}\mathrm{~d}x\\
    &\qquad+\int_{B}(a(x)-a(\bar{y}))\partial_{\varepsilon} V^{\varepsilon, y}(V^{\varepsilon, y})^{2^*-1}\mathrm{~d}x\\
    &\qquad+O\left(\varepsilon^{\frac{N-4}{2}}\right)\\
    &= O(\varepsilon^{\theta-1})+O\left(\varepsilon^{\frac{N-4}{2}}\right).
\end{align*}
\medskip
\noindent
The next term of $\delta_{\varepsilon}\tilde{A}$ is given by
\begin{align*}
(2^*-1)\int_{B} a(x) \partial_{\varepsilon} PV^{\varepsilon, y} (P  V^{\varepsilon, y})^{2^*-2}\delta v\;{\rm d}x&=(2^*-1)\int_{B} a(y) \partial_{\varepsilon} PV^{\varepsilon, y} (P  V^{\varepsilon, y})^{2^*-2}\delta v\;{\rm d}x\\
&\quad+(2^*-1)\int_{B} (a(x)-a(y)) \partial_{\varepsilon} PV^{\varepsilon, y} (P  V^{\varepsilon, y})^{2^*-2}\delta v\;{\rm d}x.
\end{align*}
Using H\"older's inequality, the second term of the above integral can be estimated as
\begin{equation*}
   (2^*-1) \int_{B}(a(x)-a(y)) \partial_{\varepsilon} PV^{\varepsilon, y} (P  V^{\varepsilon, y})^{2^*-2}\delta v\;{\rm d}x=O(\delta).
\end{equation*}
For the first term, we have
\begin{align*}
&(2^*-1)\int_{B} a(y) \partial_{\varepsilon} PV^{\varepsilon, y} (P  V^{\varepsilon, y})^{2^*-2}\delta v\;{\rm d}x\\
&= (2^*-1)\delta a(y)  \left[\underbrace{\int_{B}\partial_{\varepsilon} V^{\varepsilon, y}(V^{\varepsilon, y})^{2^*-2}v\;{\rm d}x}_{B1}-\underbrace{\int_{B}\partial_{\varepsilon} \varphi^{\varepsilon, y}(V^{\varepsilon, y})^{2^*-2} v\;{\rm d}x}_{B2} \right. \\
&\left. \hspace{0.8cm}+O\left(\underbrace{\int_{B}\left|\partial_{\varepsilon} V^{\varepsilon, y}\right|\varphi^{\varepsilon, y}(V^{\varepsilon, y})^{2^*-3}|v|\;{\rm d}x}_{B3}\right)+O\left(\underbrace{\int_{B}\left|\partial_{\varepsilon}\varphi^{\varepsilon, y}\right| \varphi^{\varepsilon, y}(V^{\varepsilon, y})^{2^*-3}|v|\;{\rm d}x}_{B4} \right)\right].
\end{align*}
Next, we estimate the integrals in the above expression.\\
For $B1$, take $v \in H_0^1(B(0,1))$ as a test function in the following equation
\begin{equation}
\begin{array}{c}
-\Delta \partial_{\varepsilon}PV^{\varepsilon, y}= (2^*-1)(V^{\varepsilon, y})^{2^*-2}\partial_{\varepsilon}V^{\varepsilon, y}  \quad \text { on } B(0,1).
\end{array}
\end{equation}
Using \eqref{OrthCondn}, we obtain 
$B1\equiv 0$.
Next, we estimate $B2$ using H\"older's inequality, the Sobolev embedding theorem, and Estimate \ref{ErrorPro} to get
\begin{align*}
    B2&:=\int_{B}\partial_{\varepsilon} \varphi^{\varepsilon, y}(V^{\varepsilon, y})^{2^*-2} v\;{\rm d}x
    =O\left(\frac{\varepsilon^{\frac{N-4}{2}}}{ d^{\frac{(N-2)}{2}}}\right).
\end{align*}
We then estimate $B3$ using $\varphi^{\varepsilon,y}\leq V^{\varepsilon,y}$:
\begin{align*}
    B3&:= \int_{B}\left|\partial_{\varepsilon} V^{\varepsilon, y}\right|\varphi^{\varepsilon, y}(V^{\varepsilon, y})^{2^*-3}|v|\;{\rm d}x\\
    &=O\left[\left\|\varphi^{\varepsilon,y}\right\|_\infty \left(\int_{B}|v|^{2^*}\right)^{\frac{1}{2^*}}\left(\int_{B_d}|\partial_{\varepsilon} V^{\varepsilon, y}|^{\frac{2^*}{2^*-1}}|V^{\varepsilon, y}|^{\frac{(2^*-3)2^*}{2^*-1}}\;{\rm d}x\right)^{\frac{2^*-1}{2^*}}\right.\\
   & \hspace{3cm}+\left. \left(\int_{B}|v|^{2^*}\right)^{\frac{1}{2^*}}\left(\int_{B\setminus B_d}|\partial_{\varepsilon} V^{\varepsilon, y}|^{\frac{2^*}{2^*-1}}|V^{\varepsilon, y}|^{\frac{(2^*-2)2^*}{2^*-1}}\;{\rm d}x\right)^{\frac{2^*-1}{2^*}}\right].
\end{align*}
Using \eqref{phiLinf}, we can deduce the following estimate for $B3$
\begin{align*}
 B3 &= O\left(\frac{\varepsilon^{\frac{N}{2}}}{d^{\frac{N+2}{2}}}+\begin{cases} \frac{\varepsilon^{\frac{N}{2}}}{d^{\frac{N+2}{2}}}\text{ if } N >6\\
\frac{\varepsilon^{3}|\log \left(\frac{d}{\varepsilon}\right)|^{\frac{2}{3}}}{d^{4}} \text{ if } N =6\\
\frac{\varepsilon^{N-3}}{d^{N-2}} \text{ if } N <6
\end{cases} \right).
\end{align*}
Finally, $B4$ can be estimated similarly to $B2$:
\begin{align*}
B4:=\int_{B}\left|\partial_{\varepsilon}\varphi^{\varepsilon, y}\right| \varphi^{\varepsilon, y}(V^{\varepsilon, y})^{2^*-3}|v|\;{\rm d}x=O\left(\frac{\varepsilon^{\frac{N-4}{2}}}{ d^{\frac{(N-2)}{2}}}\right).
\end{align*}
\noindent
Using similar techniques as in the previous estimates, we can evaluate the next term in $\partial_{\varepsilon}\tilde{A}$
  \begin{align*}
\int_{B} a(x)  \left|\partial_{\varepsilon} PV^{\varepsilon, y}\right| \left(P  V^{\varepsilon, y}\right)^{2^*-3}|\delta v|^{2}\;{\rm d}x&=O\left(\delta^2 \varepsilon^{-1}\right).
\end{align*}
We can use analogous methods to estimate the remaining term, yielding
\begin{align*}
   \int_{B} \left|\partial_{\varepsilon} P  V^{\varepsilon, y}\right||\delta v|^{2^*-1}\mathrm{~d}x=O(\varepsilon^{2^*-2} \varepsilon^{\kappa^\prime(2^*-1)}),
\end{align*}
where $\kappa^\prime$ is chosen such that $\frac{3-2^*}{2^*-1}< \kappa^\prime <1$.

\noindent
Therefore, by combining all these estimates and using the condition $\delta < 2 C^{+} \varepsilon^{1+\kappa^\prime}$, we conclude that
\begin{align*}
\partial_{\varepsilon} \tilde{J}\left(P  V^{\varepsilon, y}+\delta v\right)&=-\frac{C(N)\tilde{\lambda}a(y)^{\frac{2-N}{2}}\varepsilon}{(1-|y|^2)^2}\int_{\Rn}\frac{(|z|^2-1)}{(1+|z|^2)^{N-1}} \mathrm{~d}z+O(\varepsilon^2)+O(\varepsilon^{\theta-1})+O(\varepsilon^{\frac{N}{2}})\\
& \hspace{5cm}+O(\varepsilon^{2^*-2} \varepsilon^{\kappa^\prime(2^*-1)})+O(\delta)+O(\delta^2 \varepsilon^{-1})
\end{align*}
and the result follows when $\delta$ and $\varepsilon$ are sufficiently small.
\end{proof}
\noindent
We can now state the key proposition of this section:
\begin{proposition}\label{main-prop}
Let $N>6$. There exists $\rho>0$ such that for every $\gamma \in \Gamma$, there is a corresponding $\tilde{\gamma} \in \Gamma$ that satisfies
\begin{equation*}
\tilde{\gamma}\left(B(0,R_2)^\complement\right) \cap \mathcal{I}_\rho=\emptyset \quad \text { and } \quad \max _{y \in B(0,R_2)^\complement} \mathcal{J}_a(\tilde{\gamma}(y)) \leq \max \left(d_\lambda, \max _{y \in B(0,R_2)^\complement} \mathcal{J}_a(\gamma(t))\right) .
\end{equation*}
\end{proposition}

\medskip

\begin{proof}
Standard arguments of the deformation type lead to the proof. Indeed, the proof's essential phases depend on earlier lemmas. Together with Lemma~\ref{epsiChan}, Lemma~\ref{delta_dep}, Lemma~\ref{energyMain}, and Lemma~\ref{defo}, we adhere to the procedures of \cite[Proposition~5.9]{Sm} to obtain the desired proposition.
\end{proof}


\section{Appendix}  \label{appen}
For the proof of the following estimate, refer to \cite[Proposition 1]{ReyO}. In the following estimates, we will refer to $B(0,1)$ simply as $B$.

\begin{estimate} \label{ErrorPro}
Let $(y, \varepsilon) \in B \times \mathbf{R}_{+}$ and define $\varphi^{\varepsilon,y} = V^{\varepsilon,y} - P V^{\varepsilon,y}$. Then, we have the following:
\begin{enumerate}
\item \begin{equation*}
0 \leqslant \varphi^{\varepsilon,y} \leqslant V^{\varepsilon,y}.
\end{equation*}
\item
\begin{equation*}
\varphi^{\varepsilon,y}=\varepsilon^{\frac{N-2}{2}}H(y, \cdot)+f^{\varepsilon,y},
\end{equation*}
\noindent
where $f^{\varepsilon,y}$ satisfies the uniform estimates:
\begin{equation*}
\begin{gathered}
f^{\varepsilon,y}=O\left(\frac{\varepsilon^{\frac{N+2}{2}}}{ d^N}\right), \quad \frac{\partial f^{\varepsilon,y}}{\partial \varepsilon}=O\left(\frac{\varepsilon^{\frac{N}{2}}}{ d^N}\right), \quad \frac{\partial f^{\varepsilon,y}}{\partial y_i}=O\left(\frac{\varepsilon^{\frac{N+2}{2}}}{ d^{N+1}}\right).
\end{gathered}
\end{equation*}
\item
\begin{equation*}
\begin{gathered}
\left|\varphi^{\varepsilon,y}\right|_{L^{2^*}(B)}=O\left(\frac{\varepsilon^{\frac{N-2}{2}}}{d^{\frac{N-2}{2}}}\right),\\
\left|\frac{\partial \varphi^{\varepsilon,y}}{\partial \varepsilon}\right|_{L^{2^*}(B)}=O\left(\frac{\varepsilon^{\frac{N-4}{2}}}{ d^{\frac{(N-2)}{2}}}\right), \quad\left| \frac{\partial \varphi^{\varepsilon,y}}{\partial y_i}\right|_{L^{2^*}(B)}=O\left(\frac{\varepsilon^{\frac{N-2}{2}}}{ d^{\frac{N}{2}}}\right),
\end{gathered}
\end{equation*}
where $d=d(y, \partial B)$ is the distance between $y$ and the boundary of $B$.
\end{enumerate}
\end{estimate}
\begin{estimate}
$\int_{B}\left|\nabla P V^{\varepsilon,y}\right|^2 \; {\rm d}x = \int_{\Rn}V^{2^*} \; {\rm d}x +O\left(\left(\frac{\varepsilon}{d}\right)^{N-2}\right)$. \label{proGrad}
\end{estimate}
\begin{proof}
    \begin{align*}
      \int_{B}\left|\nabla P V^{\varepsilon,y}\right|^2 \; {\rm d}x =   \int_{B}(V^{\varepsilon,y})^{2^{*}-1}P V^{\varepsilon,y} \mathrm{~d}x =  \underbrace{\int_{B}(V^{\varepsilon,y})^{2^{*}}\mathrm{~d}x}_{A1}-\underbrace{\int_{B}(V^{\varepsilon,y})^{2^{*}-1}\varphi^{\varepsilon,y}\mathrm{~d}x}_{A2}.
    \end{align*}
   Now, we estimate $A1$ and $A2$ as follows:
     \begin{align*}
         A1:= \int_{B}(V^{\varepsilon,y})^{2^{*}} \; {\rm d}x&= \int_{\Rn}(V^{\varepsilon,y})^{2^{*}}\mathrm{~d}x-\int_{\Rn \setminus B}(V^{\varepsilon,y})^{2^{*}}\mathrm{~d}x.
         \end{align*}
     Then 
         \begin{align*}
      \int_{\Rn \setminus B}(V^{\varepsilon,y})^{2^{*}}\mathrm{~d}x\;& \leq  \int_{\Rn \setminus B(y,d)}(V^{\varepsilon,y})^{2^{*}}\mathrm{~d}x\\
      &= \int_{\Rn \setminus B(y,d)} \[\frac{\varepsilon^{\frac{N-2}{2}}}{\[\varepsilon^2+|x-y|^2\]^{\frac{N-2}{2}}}\]^\frac{2N}{N-2} \mathrm{~d}x
         &= \int_{\frac{d}{\varepsilon}}^{\infty}\frac{r^{N-1}}{(1+r^2)^{N}}\mathrm{~d}r & = O\left(\left(\frac{\varepsilon}{d}\right)^N\right).
         \end{align*}
Therefore,
         \begin{align}
&\int_{B}(V^{\varepsilon,y})^{2^{*}} \; {\rm d}x=\int_{\Rn}V^{2^{*}}\mathrm{~d}x+O\left(\left(\frac{\varepsilon}{d}\right)^N\right). \label{AubinCritPower}
     \end{align}
     Now, for $A2$
     \begin{align*}
         A2:= &\int_{B}(V^{\varepsilon,y})^{2^{*}-1}\varphi^{\varepsilon,y}\mathrm{~d}x\\
         &\leq \|\varphi^{\varepsilon,y}\|_{\infty}\int_{B}\frac{1}{\varepsilon^{\frac{N+2}{2}}\[1+\[\frac{|x-y|}{\varepsilon}\]^2\]^{\frac{N+2}{2}}}\mathrm{~d}x\\
         &\leq \frac{\varepsilon^{\frac{N}{2}-1}}{d^{N-2}}\varepsilon^{\frac{N}{2}-1}\int_{0}^{\infty}\frac{r^{N-1}}{(1+r^2)^{\frac{N+2}{2}}}\mathrm{~d}r=O\left(\left(\frac{\varepsilon}{d}\right)^{N-2}\right).
     \end{align*}
Hence, the proof is complete.
\end{proof}
\noindent
Additionally, we also have the following estimate
\begin{equation}
   \int_{B}\left|\nabla V^{\varepsilon,y}\right|^2 \; {\rm d}x = \int_{\Rn}V^{2^*} \; {\rm d}x+O\left(\left(\frac{\varepsilon}{d}\right)^{N-2}\right).\label{AubGrad} 
\end{equation}
The following estimates concern integrals involving the projection operator and its derivatives.
\begin{estimate} \label{Proj&ProjDer}
    \begin{align}
& \int_{B} \nabla P V^{\varepsilon, y} \nabla \frac{\partial P V^{\varepsilon, y}}{\partial y_{j}} \; {\rm d}x= -\varepsilon^{N-2}\frac{\partial H}{\partial y_{j}}(y, y) \int_{\Rn} V^{2^*-1} \; {\rm d}x+O\left(\frac{\varepsilon^{N}}{d^2}\right)+O\left(\frac{\varepsilon^{N}}{d^{N+1}}\right).\label{ProDely}\\
& \int_{B} \nabla P V^{\varepsilon, y}\nabla \frac{\partial P V^{\varepsilon, y}}{\partial \varepsilon}\; {\rm d}x=-\varepsilon^{N-3} H(y, y) \int_{\Rn} V^{2^*-1} \; {\rm d}x+O\left(\frac{\varepsilon^{N-1}}{d^2}\right)+O\left(\frac{\varepsilon^{N-1}}{d^{N}}\right).\label{ProjDelEp}\\
& \int_{B}\left|\nabla \frac{\partial P V^{\varepsilon, y}}{\partial y_{i}}\right|^{2} \; {\rm d}x= \frac{C}{\varepsilon^{2}}+C\varepsilon^{N-3}\frac{\partial H}{\partial y_{j}}(y, y)+O\left(\frac{\varepsilon^N}{d^3}\right)+O\left( \frac{\varepsilon^{N-2}}{d^N}\right)+O\left(\frac{\varepsilon^N}{d^{N+2}}\right)+O\left(\frac{\varepsilon^{N-1}}{d^{N+1}}\right).\label{ProjDely2}  \\
& \int_{B}\left|\nabla \frac{\partial P V^{\varepsilon, y}}{\partial \varepsilon}\right|^{2}\; {\rm d}x=\frac{C}{\varepsilon^2}+CH(y,y)\frac{\varepsilon^{N-2}}{d^2}+O\left(\frac{\varepsilon^{N-2}}{d^N}\right)+O\left(\varepsilon^{N-4}\right).\label{ProjdelEp2}\\
& \int_{B} \nabla \frac{\partial P V^{\varepsilon, y}}{\partial \varepsilon} \nabla \frac{\partial P V^{\varepsilon, y}}{\partial y_{i}} \; {\rm d}x=O\left(\frac{\varepsilon^{N-3}}{d^{N-1}}\right). \label{interterm1}\\
& \int_{B} \nabla \frac{\partial P V^{\varepsilon, y}}{\partial y_{j}} \nabla \frac{\partial P V^{\varepsilon, y}}{\partial y_{i}}\; {\rm d}x=O\left(\frac{\varepsilon^{N-2} }{d^{N}}\right), \quad i \neq j. \label{intermterm2}
\end{align}
\end{estimate}
\begin{proof}
The derivations for these estimates are detailed in (B.4) to (B.9) of \cite{ReyO}. 
\end{proof}
\noindent
However, we will utilize the following estimates that arise in the proof of these results, and hence, will cite them separately.
\begin{equation}
\begin{aligned}
\int_{B(y,d)} (V^{\varepsilon,y})^{2^{*}-1}\; {\rm d}x & = C(N) \int_{B(y,d)}\frac{\; {\rm d}x}{\varepsilon^{\frac{N+2}{2}}\[1+\[\frac{|x-y|}{\varepsilon}\]^2\]^{\frac{N+2}{2}}}= \int_{\Rn}  -\int_{\Rn \setminus B(y,d)}   \\
& =\varepsilon^{\frac{N}{2}-1}\int_{\Rn} V^{2^*-1} \; {\rm d}x+O\left(\varepsilon^{\frac{N-2}{2}} \int_{\frac{d}{\varepsilon}}^{\infty} \frac{r^{N-1} \mathrm{~d}r}{\left(1+r^{2}\right)^{(N+2) / 2}}\right) \\
& =\varepsilon^{\frac{N}{2}-1}\int_{\Rn} V^{2^*-1}\; {\rm d}x+O\left(\frac{\varepsilon^{\frac{N+2}{2}}}{ d^{2}}\right).
\end{aligned}\label{aubinCritminusone}
\end{equation}
\noindent  
The $L^{\infty}$ estimate for $\varphi^{\varepsilon,y}$ on $B$, using the maximum principle, is given by
\begin{align*}
  |\varphi^{\varepsilon,y}(x)|\leq \max_{\overline{B}}\varphi^{\varepsilon,y}(x)= \max_{\partial B}\frac{\varepsilon^{\frac{N-2}{2}}}{\left[\varepsilon^2+|x-y|^2\right]^{\frac{N-2}{2}}}\leq \frac{\varepsilon^{\frac{N-2}{2}}}{d^{N-2}},
\end{align*}
thus,
\begin{equation}
\left\|\varphi^{\varepsilon,y}\right\|_{\infty}\leq  \frac{\varepsilon^{\frac{N-2}{2}}}{d^{N-2}}. \label{phiLinf}
\end{equation}
Similarly, for the first partial derivatives,
\begin{align*}
\left|\frac{\partial \varphi^{\varepsilon,y}}{\partial y_{j}}(x)\right|\leq \max_{\overline{B}}\frac{\partial \varphi^{\varepsilon,y}}{\partial y_{j}}(x)\leq  \max_{\partial B}\frac{\varepsilon^{\frac{N-2}{2}}|x-y|}{\left[\varepsilon^2+|x-y|^2\right]^{\frac{N}{2}}}\leq \frac{\varepsilon^{\frac{N-2}{2}}}{d^{N-1}},    \end{align*}
so,
$$\left\|\frac{\partial \varphi^{\varepsilon,y}}{\partial y_{j}}\right\|_{\infty}\leq \frac{\varepsilon^{\frac{N-2}{2}}}{d^{N-1}}.$$
Similarly,
\begin{equation*}
\left\|\frac{\partial \varphi^{\varepsilon,y}}{\partial \varepsilon}\right\|_{\infty}\leq \frac{\varepsilon^{\frac{N-4}{2}}}{d^{N-2}}. \label{PhiDelEpsLinf}
\end{equation*}

\medskip
We now compute integrals involving the second derivative of the projection operator with respect to the parameters $\varepsilon$ and $y$.
\medskip
\begin{estimate} \label{DoubleDer}
 \begin{align}
& \left\|\frac{\partial^{2} P V^{\varepsilon,y}}{\partial y_{i} \partial \varepsilon}\right\|_{H_{0}^{1}(B)}=O\left(\frac{1}{\varepsilon^2}\right).  \label{ProjSeconDerEpsY}\\
& \left\|\frac{\partial^{2} P V^{\varepsilon,y}}{\partial y_{i} \partial y_{j}}\right\|_{H_{0}^{1}(B)}=O\left(\frac{1}{\varepsilon^{2}}\right). \label{ProjSeconDerY}\\
& \left\|\frac{\partial^{2} P V^{\varepsilon,y}}{\partial \varepsilon^{2}}\right\|_{H_{0}^{1}(B)}=O\left(\frac{1}{\varepsilon^{2}}\right).\label{ProjSeconDerEps}
\end{align}
   \end{estimate}
\begin{proof}
Refer to (B.19), (B.20), (B.21) of \cite{ReyO}.
\end{proof}
\noindent
Next, we turn to the estimates of the integrals involving the conformal factor.
\begin{estimate}\label{ConfTerms}For $N > 4$, the following estimates hold:
 \begin{align}
        &\int_{B} \frac{\left|P V^{\varepsilon,y}\right|^{2}}{\left(1-|x|^{2}\right)^{2}}\mathrm{~d}x=\frac{C(N)\varepsilon^2}{\left(1-|y|^2\right)^2} \int_{\mathbb{R}^{N}} \frac{1}{\left(1+|z|^{2}\right)^{N-2}} \mathrm{~d}z +O(\varepsilon^{N-2})+\begin{cases} O(\varepsilon^3) \text{ if } N >5\\
O(\varepsilon^3 |\log \varepsilon|) \text{ if } N =5.
\end{cases}\label{ConfTerm1}\\
& \int_{B} \frac{\left|\partial_\varepsilon P V^{\varepsilon,y}\right|^{2}}{\left(1-|x|^{2}\right)^{2}}\mathrm{~d}x = \frac{C(N)}{\left(1-|y|^2\right)^2}\int_{\Rn} \frac{(|z|^2-1)^2}{\left(1+|z|^{2}\right)^{N}}\mathrm{~d}z + O(\varepsilon^{N-4})+\begin{cases} O(\varepsilon) \text{ if } N >5\\
O(\varepsilon |\log \varepsilon|) \text{ if } N =5.
\end{cases}\label{ConfTerm2}\\
&\int_{B} \frac{P V^{\varepsilon,y}\partial_\varepsilon P V^{\varepsilon,y}}{\left(1-|x|^{2}\right)^{2}}\mathrm{~d}x=  \frac{C(N)\varepsilon}{\left(1-|y|^2\right)^2}\int_{\Rn} \frac{|z|^2-1}{(1+|z|^2)^{N-1}}\mathrm{~d}z+ O\left(\varepsilon^{N-3}\right)+\begin{cases} O(\varepsilon^2) \text{ if } N >5\\
O(\varepsilon^2 |\log \varepsilon|) \text{ if } N =5.
\end{cases}
\label{ConfTerm3}
    \end{align}
\end{estimate}
\begin{proof}
We begin by proving \eqref{ConfTerm1} as follows
    \begin{align}
       \int_{B} \frac{\left|P V^{\varepsilon,y}\right|^{2}}{\left(1-|x|^{2}\right)^{2}}\mathrm{~d}x= \underbrace{\int_{B\left(y,\frac{d}{4}\right)}}_{I1}+\underbrace{\int_{B\setminus B\left(y,\frac{d}{4}\right)}}_{I2}.
    \end{align}
We first consider the integral $I1$
\begin{equation*}
    \begin{aligned}
     I1= \int_{B\left(y,\frac{d}{4}\right)} \frac{\left|P V^{\varepsilon,y}\right|^{2}}{\left(1-|x|^{2}\right)^{2}}\mathrm{~d}x =   \underbrace{\int_{B\left(y,\frac{d}{4}\right)} \frac{\left(V^{\varepsilon,y}\right)^{2}}{\left(1-|x|^{2}\right)^{2}}\mathrm{~d}x}_{I1.a}+\underbrace{\int_{B\left(y,\frac{d}{4}\right)} \frac{\left(\varphi^{\varepsilon,y}\right)^{2}}{\left(1-|x|^{2}\right)^{2}}\mathrm{~d}x}_{I1.b}-\underbrace{2\int_{B\left(y,\frac{d}{4}\right)} \frac{V^{\varepsilon,y} \varphi^{\varepsilon,y}}{\left(1-|x|^{2}\right)^{2}}\mathrm{~d}x}_{I1.c}.
    \end{aligned}
\end{equation*}
Let's first address $I1.a$
$$
\begin{aligned}
I1.a:=\int_{B\left(y,\frac{d}{4}\right)} \frac{\left(V^{\varepsilon,y}\right)^{2}}{\left(1-|x|^{2}\right)^{2}}\; {\rm d}x= &\;C(N) \int_{B\left(y,\frac{d}{4}\right)} \frac{\varepsilon^{N-2}}{\left(\varepsilon^{2}+\left|x-y\right|^{2}\right)^{N-2}\left(1-|x|^{2}\right)^{2}} \mathrm{~d}x\\
& =C(N) \varepsilon^{2} \int_{|z| < \frac{d}{4\varepsilon}} \frac{1}{\left(1+|z|^{2}\right)^{N-2}\left(1-\left|z \varepsilon+y\right|^{2}\right)^{2}} \mathrm{~d}z.
\end{aligned}
$$
Next, consider the following integral
\begin{align*}
& \int_{|z| < \frac{d}{4\varepsilon}} \frac{1}{\left(1+|z|^{2}\right)^{N-2}}\left[\frac{1}{\left(1-\left|z \varepsilon+y\right|^{2}\right)^{2}}-\frac{1}{\left(1-\left|y\right|^{2}\right)^{2}}\right] \mathrm{~d}z \\
& =\int_{|z| < \frac{d}{4\varepsilon}} \frac{1}{\left(1+|z|^{2}\right)^{N-2}}\left[\frac{\left(1-\left|y\right|^{2}\right)^{2}-\left(1-\left|z \varepsilon+y\right|^{2}\right)^{2}}{\left(1-\left|z \varepsilon+y\right|^{2}\right)^{2}\left(1-|y|^{2}\right)^{2}}\right] \mathrm{~d}z \\
& =\int_{|z| < \frac{d}{4\varepsilon}}\left[\frac{\left(\varepsilon^{2}|z|^{2}+2 z \varepsilon \cdot y\right)\left(1-\left|y\right|^{2}\right)+\left(1-\left|z \varepsilon+y\right|^{2}\right)\left(\varepsilon^{2}|z|^{2}+2 z \varepsilon \cdot y\right)}{\left(1+|z|^{2}\right)^{N-2}\left(1-\left|z \varepsilon+y\right|^{2}\right)^{2}\left(1-\left|y\right|^{2}\right)^{2}}\right]\mathrm{~d}z  \\
& =\int_{|z| < \frac{d}{4\varepsilon}} \frac{1}{\left(1+|z|^{2}\right)^{N-2}} \frac{\left(\varepsilon^{2}|z|^{2}+2 z \varepsilon \cdot y\right)}{\left(1-\left|z \varepsilon+y\right|^{2}\right)^{2}\left(1-\left|y\right|^{2}\right)}\mathrm{~d}z  \\
& \hspace{3cm}+\int_{|z| < \frac{d}{4\varepsilon}} \frac{1}{\left(1+|z|^{2}\right)^{N-2}} \frac{(\varepsilon^{2}|z|^{2}+2 z \varepsilon \cdot y)}{\left(1-\left|z \varepsilon+y\right|^{2}\right)\left(1-\left|y\right|^{2}\right)^{2}}\mathrm{~d}z  \\
& \leq \frac{c \varepsilon^{2}}{\left(1-\left|y\right|^{2}\right)^2\left(1-\frac{d^{2}}{16}-\frac{d}{2}|y|-\left|y\right|^{2}\right)^{2}} \int_{|z| < \frac{d}{4\varepsilon}} \frac{|z|^2}{\left(1+|z|^{2}\right)^{N-2}}\mathrm{~d}z \\
& \hspace{3cm} + \frac{c\varepsilon\left|y\right|}{\left(1-\left|y\right|^{2}\right)^2\left(1-\frac{d^{2}}{16}-\frac{d}{4}|y|-\left|y\right|^2\right)^{2}} \int_{|z| < \frac{d}{4\varepsilon}} \frac{|z|}{\left(1+|z|^{2}\right)^{N-2}} \mathrm{~d}z \\
& \leq \frac{c \varepsilon^{2}}{\left(1-\left|y\right|^{2}\right)^2\left(1-\frac{d^{2}}{16}-\frac{d}{2}|y|-\left|y\right|^2\right)^{2}} \int_{0}^{\frac{d}{4\varepsilon}} \frac{r^{N+1}}{\left(1+r^{2}\right)^{N-2}}\mathrm{~d}r \\
& \hspace{3cm}+\frac{c \varepsilon}{\left(1-\left|y\right|^{2}\right)^2\left(1-\frac{d^{2}}{16}-\frac{d}{2}|y|-\left|y\right|^2\right)^{2}} \int_{0}^{\frac{d}{4\varepsilon}} \frac{r^{N}}{\left(1+r^{2}\right)^{N-2}} \mathrm{~d}r,
\end{align*}
and
$$
\begin{aligned}
& \int_{0}^{\frac{d}{4\varepsilon}} \frac{r^{N+1}}{\left(1+r^{2}\right)^{N-2}}\mathrm{~d}r=\begin{cases}
O(1) \text { if } N>6, \\
O(|\log \varepsilon|) \text { if } N=6, \\
O\left(\varepsilon^{-1}\right) \text { if } N=5,\\
O\left(\varepsilon^{-2}\right) \text { if } N=4,
\end{cases} \\
& \int_{0}^{\frac{d}{4\varepsilon}} \frac{r^{N}}{\left(1+r^{2}\right)^{N-2}}\mathrm{~d}r=\begin{cases}
O(1) \text { if } \quad N>5, \\
O(|\log \varepsilon|) \text { if } N=5,\\
O(\varepsilon^{-1}) \text { if } N=4.
\end{cases}
\end{aligned}
$$
Thus, we have
$$
\int_{|z|<\frac{d}{4\varepsilon}} \frac{1}{\left(1+|z|^{2}\right)^{N-2}}\left[\frac{1}{\left(1-\left|z \varepsilon+y\right|^{2}\right)^{2}}-\frac{1}{\left(1-\left|y\right|^{2}\right)^{2}}\right] \mathrm{~d}z=\begin{cases} O(\varepsilon) \text{ if } N >5,\\
O(\varepsilon |\log \varepsilon|) \text{ if } N =5,\\
O(1) \text{ if } N =4,
\end{cases}
$$
as $\varepsilon \rightarrow 0$. Consequently, we obtain
\begin{align*}
& \int_{|z| < \frac{d}{4\varepsilon}} \frac{1}{\left(1+|z|^{2}\right)^{N-2}} \frac{1}{\left(1-\left|z \varepsilon+y\right|^{2}\right)^{2}} \mathrm{~d}z \\
= & \frac{1}{\left(1-\left|y\right|^{2}\right)^{2}} \int_{|z| < \frac{d}{4\varepsilon}} \frac{1}{\left(1+|z|^{2}\right)^{N-2}} \mathrm{~d}z+\begin{cases} O(\varepsilon) \text{ if } N >5,\\
O(\varepsilon |\log \varepsilon|) \text{ if } N =5,\\
O(1) \text{ if } N =4,
\end{cases} \\
= & \frac{1}{\left(1-\left|y\right|^{2}\right)^{2}} \int_{\mathbb{R}^{N}} \frac{1}{\left(1+|z|^{2}\right)^{N-2}} \mathrm{~d}z- \underbrace{\frac{1}{\left(1-\left|y\right|^{2}\right)^{2}} \int_{\mathbb{R}^{N}\setminus B(0, \frac{d}{4\varepsilon})} \frac{1}{\left(1+|z|^{2}\right)^{N-2}} \mathrm{~d}z}_{\int_{\frac{d}{4\varepsilon}}^{\infty}\frac{r^{N-1}}{(1+r^2)^{N-2}}\mathrm{~d}r= O\left(\frac{\varepsilon^{N-4}}{d^{N-4}}\right)}+\begin{cases} O(\varepsilon) \text{ if } N >5,\\
O(\varepsilon |\log \varepsilon|) \text{ if } N =5,\\
O(1) \text{ if } N =4,
\end{cases}\\
=& \frac{1}{\left(1-\left|y\right|^{2}\right)^{2}} \int_{\mathbb{R}^{N}} \frac{1}{\left(1+|z|^{2}\right)^{N-2}} \mathrm{~d}z+\begin{cases} O(\varepsilon) \text{ if } N >5,\\
O(\varepsilon |\log \varepsilon|) \text{ if } N =5,\\
O(1) \text{ if } N =4.
\end{cases}
\end{align*}
As a result, we have
\begin{equation*}
\begin{aligned}
I1.a:=   \frac{C(N)\varepsilon^2}{\left(1-\left|y\right|^{2}\right)^{2}} \int_{\mathbb{R}^{N}} \frac{1}{\left(1+|z|^{2}\right)^{N-2}} \mathrm{~d}z +\begin{cases} O(\varepsilon^3) \text{ if } N >5,\\
O(\varepsilon^3 |\log \varepsilon|) \text{ if } N =5,\\
O(\varepsilon^2) \text{ if } N =4.
\end{cases} 
\end{aligned}
\end{equation*}
Note that $(1 - |x|^2) > C > 0$ for all $x \in B\left(y, \frac{d}{4}\right)$. Moreover, using \eqref{phiLinf}, we find 
\begin{equation*}
\begin{aligned}
I1.b:=& \int_{B\left(y,\frac{d}{4}\right)} \frac{(\varphi^{\varepsilon, y})^{2}}{\left(1-|x|^{2}\right)^{2}} \mathrm{~d}x \\
& \leq \left\|\varphi^{\varepsilon,y}\right\|^2_{\infty}\int_{B\left(y,\frac{d}{4}\right)} \frac{1}{\left(1-|x|^{2}\right)^{2}}\mathrm{~d}x\leq C \frac{\varepsilon^{N-2}}{d^{2(N-2)}} d^N = O\left(\varepsilon^{N-2}\right).
\end{aligned}
\end{equation*}
Finally, we can estimate the third integral as follows
\begin{equation*}
\begin{aligned}
I1.c:= \int_{B\left(y,\frac{d}{4}\right)} \frac{V^{\varepsilon,y} \varphi^{\varepsilon, y}}{\left(1-|x|^{2}\right)^{2}} \mathrm{~d}x 
&\leq C\left\|\varphi^{\varepsilon,y}\right\|_{\infty}\int_{B\left(y,\frac{d}{4}\right)} V^{\varepsilon,y}\mathrm{~d}x\\
&\leq C \frac{\varepsilon^{\frac{N-2}{2}}}{d^{N-2}}\int_{B\left(y,\frac{d}{4}\right)}\frac{1}{\varepsilon^{\frac{N-2}{2}}\left[1+\frac{|x-y|^2}{\varepsilon^2}\right]^{\frac{N-2}{2}}}\mathrm{~d}x= O(\varepsilon^{N-2}).
\end{aligned}
\end{equation*}
Now, consider the next integral
\begin{equation*}
\int_{B\setminus B\left(y,\frac{d}{4}\right)} \frac{\left(P V^{\varepsilon,y}\right)^{2}}{\left(1-|x|^{2}\right)^{2}} \mathrm{~d}x \leq C\int_{B\setminus B\left(y,\frac{d}{4}\right)} \frac{\left(P V^{\varepsilon,y}\right)^{2}(1-\psi)^2}{\left(1-|x|^{2}\right)^{2}} \mathrm{~d}x\;+\;C\int_{B\setminus B\left(y,\frac{d}{4}\right)} \frac{\left(P V^{\varepsilon,y}\right)^{2}\psi^2}{\left(1-|x|^{2}\right)^{2}} \mathrm{~d}x,
\end{equation*}
where $\psi$ is a $C_c^{\infty}$ function that equals $1$ on $B(y, d/2)$ and vanishes outside $B(y, d)$. Applying Hardy's inequality to the first integral and considering the support of $\psi$ in the second integral, it is not difficult to conclude that $I2 = O\left(\varepsilon^{N-2}\right)$. Thus, by combining all the estimates, we obtain \eqref{ConfTerm1}.\\
The estimate \eqref{ConfTerm2} follows the same approach as the above, so we omit the proof.\\
To prove \eqref{ConfTerm3}, we start by using the identity $P V^{\varepsilon,y} = V^{\varepsilon,y} - \varphi^{\varepsilon,y}$. Applying similar techniques as before and utilizing H\"older's inequality for the integral outside $B\left(y, \frac{d}{4}\right)$, we obtain
\begin{align*}
    \int_{B} \frac{P V^{\varepsilon,y}\partial_\varepsilon P V^{\varepsilon,y}}{\left(1-|x|^{2}\right)^{2}}\mathrm{~d}x= \int_{B\left(y,\frac{d}{4}\right)}\frac{V^{\varepsilon,y}\partial_\varepsilon V^{\varepsilon,y}}{\left(1-|x|^{2}\right)^{2}}\mathrm{~d}x+ O(\varepsilon^{N-3}).
\end{align*}
To estimate the integral in the above expression, we can use the explicit form of $V^{\varepsilon,y}$ and perform the calculations as in the proof of \eqref{ConfTerm1}, leading to 
\begin{align*}
    \int_{B\left(y,\frac{d}{4}\right)}\frac{V^{\varepsilon,y}\partial_\varepsilon V^{\varepsilon,y}}{\left(1-|x|^{2}\right)^{2}}\mathrm{~d}x= \frac{C(N)\varepsilon}{(1-|y|^2)^2}\int_{\Rn} \frac{|z|^2-1}{(1+|z|^2)^{N-1}}\mathrm{~d}z+\begin{cases} O(\varepsilon^2) \text{ if } N >5,\\
O(\varepsilon^2 |\log \varepsilon|) \text{ if } N =5.
\end{cases} 
\end{align*}
Moreover, note that $\int_{\mathbb{R}^N} \frac{|z|^2-1}{(1+|z|^2)^{N-1}} \mathrm{~d}z > 0$ for $N > 4$.
  \end{proof}     
 \noindent 
Next, we proceed with the estimates of the integrals involving $v$.
\begin{estimate} \label{vTerm1}

\begin{align*}
\int_{B} PV^{\varepsilon,y} v & = \begin{cases}
O\left(\varepsilon^{\frac{1}{2}}\right) & \text { if } N=3,\\
O\left(\varepsilon \right) & \text { if } N=4,\\
O\left(\varepsilon^{\frac{3}{2}}\right) & \text { if } N=5,\\
O\left(\varepsilon^2 \left(\log \left(\frac{1}{\varepsilon}\right)\right)^{2 / 3}\right) & \text { if } N=6,\\
O\left(\varepsilon^{2}\right) & \text { if } N>6 .
\end{cases} 
\end{align*}

\end{estimate}

\begin{proof}
By applying H\"older's inequality and the Sobolev embedding theorem, we get
\begin{align*}
   \left|\int_{B} PV^{\varepsilon,y} v\right|\leq \int_{B} V^{\varepsilon,y}|v|= O\left(\left(\int_{B} (V^{\varepsilon,y})^{\frac{2^*}{2^*-1}}\right)^{\frac{2^*-1}{2^*}}\|v\|_{H_{0}^{1}(B)}\right).
   \end{align*}
 From here, using the explicit form of $V^{\varepsilon,y}$, it is straightforward to arrive at our result.
\end{proof}

\begin{estimate}For $y$ as in Lemma \ref{delta_dep} and $N>5$, we have
    \begin{align}
        \int_{B} a(x) (PV^{\varepsilon,y})^{2^*-1}v \mathrm{~d}x= O(\varepsilon^{\frac{\theta(N+2)}{2N}}). \label{vTerm2}
    \end{align}
    \begin{proof}
    We can decompose the integral as follows
        \begin{align*}
            \int_{B} a(x) (PV^{\varepsilon,y})^{2^*-1}v\mathrm{~d}x= \underbrace{a(\bar y) \int_{B} (PV^{\varepsilon,y})^{2^*-1}v\mathrm{~d}x}_{T1}+\underbrace{\int_{B} (a(x)-a(\bar y)) (PV^{\varepsilon,y})^{2^*-1}v\mathrm{~d}x}_{T2}.
        \end{align*}
        Before estimating these two terms, observe that by using \eqref{OrthCondn} and integration by parts, we obtain
        \begin{align*}
            \int_{B} (V^{\varepsilon,y})^{2^*-1}v=  \int_{B} (-\Delta V^{\varepsilon,y}) v = \int_{B} \nabla  V^{\varepsilon,y}. \nabla v =0.
        \end{align*}
   Thus, using this result and noting that $PV^{\varepsilon,y} \leq V^{\varepsilon,y}$, we can rewrite $T1$ as follows
        \begin{align*}
       &     T1:= a(\bar y) \int_{B} (PV^{\varepsilon,y})^{2^*-1}v\mathrm{~d}x= a(\bar y) \int_{B} \left[(PV^{\varepsilon,y})^{2^*-1}\mathrm{~d}x- (V^{\varepsilon,y})^{2^*-1}\right]v\mathrm{~d}x\\
&=a(\bar y)\underbrace{\int_{B(y,d)}\left[(PV^{\varepsilon,y})^{2^*-1}- (V^{\varepsilon,y})^{2^*-1}\right]v\mathrm{~d}x}_{T1.a}+ \underbrace{O\left(\int_{B\setminus B(y,d)}(V^{\varepsilon,y})^{2^*-1}|v|\mathrm{~d}x\right)}_{T1.b}.
        \end{align*}
   Observe that by applying H\"older's inequality, $T1.b$ can be estimated as follows 
   \begin{align*}
         T1.b= O\left(\left(\int_{B\setminus B(y,d)} (V^{\varepsilon,y})^{2^*}\mathrm{~d}x\right)^{\frac{2^*-1}{2^*}}\|v\|_{H_{0}^{1}(B)}\right) = O\left(\frac{\varepsilon^{\frac{N+2}{2}}}{d^{\frac{N+2}{2}}}\right).
     \end{align*}
     Next, we estimate $T1.a$
     \begin{align*}
    T1.a:=\int_{B(y,d)}\left[(PV^{\varepsilon,y})^{2^*-1}- (V^{\varepsilon,y})^{2^*-1}\right]v\mathrm{~d}x=O\left(\int_{B(y,d)}(V^{\varepsilon,y})^{2^*-2}\varphi^{\varepsilon,y}v\mathrm{~d}x\right).
     \end{align*}
   For the integral in the expression above, we use \eqref{phiLinf} to estimate
    \begin{align*}
        \int_{B(y,d)}(V^{\varepsilon,y})^{2^*-2}\varphi^{\varepsilon,y}v \mathrm{~d}x& \leq \frac{\varepsilon^{\frac{N-2}{2}}}{d^{N-2}}\left(\int_{B(y,d)} (V^{\varepsilon,y})^{(2^*-2)\frac{2^*}{2^*-1}}\mathrm{~d}x\right)^{\frac{2^*-1}{2^*}}\|v\|_{H_{0}^{1}(B)}\\
        &\leq C \frac{\varepsilon^{\frac{N-2}{2}}}{d^{N-2}} \left(\varepsilon^{\frac{N(N-2)}{N+2}}\int_{0}^{\frac{d}{\varepsilon}}\frac{r^{N-1}}{(1+r^2)^{\frac{4N}{N+2}}}\mathrm{~d}r\right)^{\frac{N+2}{2N}}\\
&=\begin{cases}
O\left(\frac{\varepsilon}{d}\right) & \text { if } N=3,\\
O\left(\frac{\varepsilon^2}{d^2}\right) & \text { if } N=4,\\
O\left(\frac{\varepsilon^3}{d^3}\right) & \text { if } N=5,\\
O\left(\frac{\varepsilon^4 \left(\log \left(\frac{d}{\varepsilon}\right)\right)^{2 / 3}}{d^4}\right) & \text { if } N=6,\\
O\left(\frac{\varepsilon^{\frac{N+2}{2}}}{d^{\frac{N+2}{2}}}\right) & \text { if } N>6.
\end{cases} 
    \end{align*}
    Now, let's estimate $T2$
        \begin{align*}
        \left|\int_{B} (a(x)-a(\bar{y}))(PV^{\varepsilon,y})^{2^*-1}v\mathrm{~d}x\right|\leq \int_{B} \left|a(x)-a(\bar{y})\right|\left(V^{\varepsilon,y}\right)^{2^*-1}|v|\mathrm{~d}x= \int_{B(y,d)}+\int_{B\setminus B(y,d)}.
        \end{align*}
       For the integral over $B(y,d)$, consider
        \begin{align*}
        \int_{B(y,d)} \left|a(x)-a(\bar{y})\right|\left(V^{\varepsilon,y}\right)^{2^*-1}|v|\mathrm{~d}x= \int_{B(y,d)\bigcap B(\bar y,r)}+\int_{B(y,d)\bigcap {B(\bar y,r)}^\complement}.
        \end{align*}
        For the integral near the point $\bar{y}$, using the assumption \ref{A2}, we get:
        \begin{align*}
        &\int_{B(y,d)\bigcap B(\bar y,r)}\left|a(x)-a(\bar{y})\right|\left(V^{\varepsilon,y}\right)^{2^*-1}|v|\mathrm{~d}x\\
        &\leq \left[\int_{B(y,d)\bigcap B(\bar y,r)}|a(x)-a(\bar{y})|^\frac{2N}{N+2}(V^{\varepsilon, y})^{2^*}\mathrm{~d}x\right]^{\frac{2^*-1}{2^*}} \left[\int_{B}|v|^{2^*}\right]^\frac{1}{2^*}\\
         &\leq  C\left[\int_{\{|z\varepsilon+y-\bar y|<r\}}|z\varepsilon+y-\bar y|^{\theta\frac{2N}{N+2}}\frac{1}{\left(1+|z|^2\right)^N} \mathrm{~d}z\right]^{\frac{2^*-1}{2^*}}\\
         &\leq C \left[r^{\frac{2N\theta}{N+2}-\theta}\varepsilon^\theta\int_{\Rn}\frac{|z|^\theta}{(1+|z^2|)^N}\mathrm{~d}z+r^{\frac{2N\theta}{N+2}-\theta}|y-\bar{y}|^\theta\int_{\Rn}\frac{1}{(1+|z|^2)^N}\mathrm{~d}z\right]^{\frac{2^*-1}{2^*}}=O(\varepsilon^{\frac{\theta(N+2)}{2N}}),
    \end{align*}
   as $\theta < N$. Finally, since $\left\{\left|\varepsilon z + y - \bar y\right| \geq r\right\} \subset {|z| \geq r / 2 \varepsilon}$ for $\left|y - \bar y\right|$ small enough, we have
    \begin{align*}
        \int_{B(y,d)\bigcap {B(\bar y,r)}^\complement}\left|a(x)-a(\bar{y})\right|\left(V^{\varepsilon,y}\right)^{2^*-1}|v|\mathrm{~d}x\leq C\|a\|_{\infty}\left[\int_{\{|z\varepsilon+y-\bar y|\geq r\}}\frac{1}{(1+|z|^2)^N}\mathrm{~d}z\right]^{\frac{2^*-1}{2^*}}=O(\varepsilon^\frac{N+2}{2}).
    \end{align*}
   Similarly, we find that $\int_{B\setminus B(y,d)}\left|a(x)-a(\bar{y})\right|\left|V^{\varepsilon,y}\right|^{2^*-1}|v|\mathrm{~d}x=O\left(\frac{\varepsilon^\frac{N+2}{2}}{d^\frac{N+2}{2}}\right)$.
    \noindent
    Combining all these estimates, we can now conclude the result.
    \end{proof}
\end{estimate}

\medskip

{\bf Acknowledgement:} The research of M.~Bhakta and A. K.~Sahoo are partially supported by Bhakta's grant DST Swarnajaynti fellowship (SB/SJF/2021-22/09), D.~Ganguly is partially supported by the SERB MATRICS (MTR/2023/000331).  and D.~Gupta is supported by the PMRF.

\end{document}